\documentclass[12pt,reqno]{amsart}
\usepackage{amssymb}
\usepackage{mathrsfs}
\usepackage{palatino}
\usepackage{bbm}
\def\HC{{\operatorname{HC}}}

\newcommand{\CC}{\mathbb{C}}

\newcommand{\NN}{\mathbb{N}}

\newcommand{\ZZ}{\mathbb{Z}}

\newcommand{\fg}{\mathfrak{g}}

\newcommand{\fl}{\mathfrak{l}}

\newcommand{\kk}{\mathbbm{k}}

\newcommand{\gl}{\mathfrak{gl}}

\newcommand{\fz}{\mathfrak{z}}

\newcommand{\Ymn}{Y_{m|n}}
\newcommand{\SYmn}{SY_{m|n}}

\DeclareMathOperator{\ad}{ad}

\DeclareMathOperator{\Char}{char}

\DeclareMathOperator{\gr}{gr}

\DeclareMathOperator{\sgn}{sgn}

\DeclareMathOperator{\halfa}{{\textstyle\frac{(-1)^{|i|}}{2}}}
\DeclareMathOperator{\halfb}{{\textstyle\frac{(-1)^{|i+1|}}{2}}}
\usepackage{mathrsfs}
\usepackage{cases}
\usepackage{epic}
\usepackage{amsfonts}
\usepackage{graphicx}
\usepackage{amsmath}
\usepackage{amssymb, upgreek}
\usepackage{bm}
\usepackage{latexsym,todonotes}
\usepackage{pdflscape}
\usepackage[all]{xypic}
\usepackage[all]{xy}
\usepackage{color}
\usepackage{colordvi}
\usepackage{multicol}
\usepackage[normalem]{ulem}
\usepackage[linktocpage=true]{hyperref}
\usepackage{hyperref}
\hypersetup{colorlinks,linkcolor=blue,urlcolor=cyan,citecolor=red}
\numberwithin{equation}{section}
\newtheorem{Theorem}{Theorem}[section]
\newtheorem{Lemma}[Theorem]{Lemma}
\newtheorem{Corollary}[Theorem]{Corollary}
\newtheorem{Proposition}[Theorem]{Proposition}

\theoremstyle{Theorem}

\newtheorem*{thm*}{Theorem}
\newtheorem*{thm**}{Corollary}
\newtheorem*{thm***}{Theorem B}
\theoremstyle{remark}
\newtheorem{Remark}{Remark}

\numberwithin{equation}{section}
\begin{document}
\title{The center of the modular super Yangian $Y_{m|n}$}
\author[Hao Chang \lowercase{and} Hongmei Hu]{Hao Chang \lowercase{and} Hongmei Hu*}
\address[Hao Chang]{School of Mathematics and Statistics, and Hubei Key Laboratory of Mathematical Sciences, Central China Normal University, 430079 Wuhan, People's Republic of China}
\email{chang@ccnu.edu.cn}
\address[Hongmei Hu]{School of Mathematical Sciences, Suzhou University of Science and Technology, 215009 Suzhou, People's Republic of China}
\email{hmhu0124@126.com}
\date{\today}
\thanks{* Corresponding author.}

\makeatletter
\makeatother
%
\begin{abstract}
The present paper is devoted to studying the super Yangian $\Ymn$
associated to the general linear Lie superalgebra $\mathfrak{gl}_{m|n}$ over a field of positive characteristic.
We extend Drinfeld-type presentations of $\Ymn$ and the special super Yangian $\SYmn$ to positive characteristic.
Moreover, the center $Z(\Ymn)$ of $\Ymn$ is described:
it is generated by its {\it Harish-Chandra center} together with a large {\it $p$-center}.
We also study the $p$-center of $\SYmn$ and provide another description of the $p$-center of $\Ymn$ in terms of the {\it RTT generators}.
\end{abstract}
\maketitle
\setcounter{tocdepth}{2}
\tableofcontents
\section{Introduction}
For each simple finite-dimensional Lie algebra $\fg$ over the field of complex numbers,
the corresponding Yangian $Y(\fg)$ was defined by Drinfeld in \cite{D85}
as a canonical deformation of the universal enveloping algebra $U(\fg[x])$ for the {\it current Lie algebra} $\fg[x].$
In \cite{D88}, Drinfeld gave a new presentation for Yangians.
The Yangians form a remarkable family of quantum groups related to rational solutions of the classical {\it Yang-Baxter equation}.
The Yangian $Y_n=Y(\mathfrak{gl}_n)$ of the Lie
algebra $\mathfrak{gl}_n$ was earlier considered in the works of mathematical physicists from
St.-Petersburg; see for instance \cite{TF79}.
It is an associative algebra whose defining
relations can be written in a specific matrix form,
which is called the {\it RTT relation};
see e.g. \cite{MNO96}.

The super Yangian $Y_{m|n}$ associated to the general linear Lie superalgebra $\mathfrak{gl}_{m|n}$ over the complex field
was defined by Nazarov \cite{Na91} in terms of the {\it RTT presentation} as a super analogue of $Y_n$.
It has been studied by several authors. The {\it Drinfeld-type presentation} was found by Gow in \cite{Gow07}.
In \cite{Peng11,Peng16}, Peng obtained the superalgebra generalization of the {\it parabolic presentations} of \cite{BK05}.
Recently, Tsymbaliuk \cite{Tsy20} provided a generalization of $\Ymn$ to arbitrary parity sequences.

In characteristic zero, the center $Z(Y_n)$ is generated by the coefficients of the {\it quantum determinant},
see \cite[Theorem 2.13]{MNO96}.
In his article \cite{Na91}, Nazarov defined the {\it quantum Berezinian} which
plays a similar role in the study of the super Yangian $\Ymn$ as the quantum determinant does in the case of $Y_n$.
Later, Gow used the Drinfeld presentation to determine the generators of the center of $Y_{m|n}$.
More precisely, the center $Z(\Ymn)$ can be generated freely by the elements $\{c^{(r)};~r>0\}$,
where the elements $c^{(r)}$ are the coefficients of the quantum Berezinian (\cite{Na91}, \cite[Theorem 2]{Gow05}, \cite[Theorem 4]{Gow07}).

In \cite{BT18},
Brundan and Topley developed the theory of the Yangian $Y_n$ over a field of positive characteristic.
In particular,
they gave a description of the center $Z(Y_n)$ of $Y_n$.
One of the key features which differs from characteristic zero is the existence of a large central subalgebra $Z_p(Y_n)$,
called the {\it $p$-center}.
Also, these results give important applications to the theory of modular finite $W$-algebras (see \cite{GT19, GT21}).

The main goal of this article is to obtain the superalgebra generalization of the $p$-center
of \cite{BT18} for the modular super Yangian of type $A$.
We define the super Yangian $\Ymn$ over an algebraically closed field $\kk$ of positive characteristic
to be the associative superalgebra by the usual RTT presentation from \cite{Na91}.
To describe the center $Z(\Ymn)$,
we found that it was easier to work initially with the Drinfeld-type presentation.
Thus the first step is to establish the modular version of Drinfeld-type presentation for $\Ymn$.
Since we are in characteristic $p:=\Char{\kk}>0$,
the modular phenomenon could happen everywhere.
One needs extra care when treating the issues arising from odd elements.
Actually, we have more relations than \cite{Gow07} (see Theorem \ref{theorem Drinfeld presentation}).
Meanwhile, comparing with the Yangian $Y_n$,
we need to insert the necessary sign factors in almost every formula.

In characteristic zero, it is well-known (\cite[Corollary 1]{Gow07}) that $\Ymn$ is a deformation of the universal enveloping algebra
$U(\mathfrak{gl}_{m|n}[x])$ for the current Lie superalgebra $\mathfrak{gl}_{m|n}[x]$.
It is clear that the element $z_r:=(e_{1,1}+\cdots+e_{m+n,m+n})\otimes x^r$ belongs to the center $Z(\mathfrak{gl}_{m|n}[x]):=Z(U(\mathfrak{gl}_{m|n}[x]))$.
Moreover, the coefficient of quantum Berezinian $c^{(r+1)}$ is a lift of $z_r$.
These definitions make sense when $\Char\kk>0$.
The algebra generated by the coefficients $\{c^{(r)}; r >0\}$ will be denoted by $Z_\HC(\Ymn)$ which is a subalgebra of the center $Z(Y_{m|n})$.
We call it the {\it Harish-Chandra center} of $\Ymn$.
Suppose that $p=\Char\kk>0$.
The current Lie superalgebra $\mathfrak{gl}_{m|n}[x]$ admits a natural structure of {\it restricted Lie superalgebra}.
That is, the even subalgebra $(\mathfrak{gl}_{m|n}[x])_{\bar 0}$ is a
restricted Lie algebra with the $p$-map $(\mathfrak{gl}_{m|n}[x])_{\bar 0}\rightarrow(\mathfrak{gl}_{m|n}[x])_{\bar 0}$
sending $a\mapsto a^{[p]}$,
and the odd part $(\mathfrak{gl}_{m|n}[x])_{\bar 1}$ is a restricted module by the adjoint action of the even subalgebra.
Then for each even element $a\in(\mathfrak{gl}_{m|n}[x])_{\bar 0}$, the element $a^p-a^{[p]}\in U(\mathfrak{gl}_{m|n}[x])$ is central.
Denote by $Z_p(\mathfrak{gl}_{m|n}[x])$ the subalgebra of $Z(\mathfrak{gl}_{m|n}[x])$ generated by all $a^p-a^{[p]}$ with $a\in(\mathfrak{gl}_{m|n}[x])_{\bar 0}$.
This subalgebra is often called the {\it $p$-center} of $U(\mathfrak{gl}_{m|n}[x])$.
It is natural to look for lifts of the $p$-central elements in $Z(\Ymn)$.
In Section 4, we will give a description of the {\it $p$-center} $Z_p(\Ymn)$ and
show that the generators of $Z_p(\Ymn)$ provide the lifts of generators for $Z_p(\mathfrak{gl}_{m|n}[x])$.
With this information in hand,
we show in particular that the center $Z(\Ymn)$ is generated by $Z_\HC(\Ymn)$ and $Z_p(\Ymn)$.

We organize this article in the following manner.
In Section \ref{section: current Lie superalgebra},
we introduce the current Lie superalgebra and determine the center of its universal enveloping algebra.
Following Nazarov's definition, we define the modular super Yangian $\Ymn$ in the RTT realization in Section \ref{section: modular superYangians}.
After recalling some basic properties of $\Ymn$,
we extend the Drinfeld-type presentation from characteristic zero to positive characteristic.
Section \ref{Section:center} is concerned with the center of $\Ymn$.
We investigate various $p$-central elements by employing the Drinfeld presentation.
Moreover, we give a description of the center of $\Ymn$ and obtain the precise formulas for the generators.
Section \ref{section:special super Yangian} is devoted to
the study about the special super Yangian $\SYmn$, which may be viewed as the modular version of the super Yangian for the Lie superalgebra $\mathfrak{sl}_{m|n}$. In particular, we obtain another description of the $p$-center of $\Ymn$ in terms of the RTT generators.

\emph{Throughout this paper, $\kk$ denotes an algebraically closed field of characteristic $\Char(\kk)=:p>0$}.

\bigskip
\section{The current superalgebra}\label{section: current Lie superalgebra}

\subsection{The current superalgebra}
Let $\gl_{m|n}[x]$ denote the {\it current superalgebra} is defined to be the Lie superalgebra $\gl_{m|n}[x]:=\gl_{m|n}\otimes\kk[x]$.
We will always denote the Lie algebra by $\fg$ and write $U(\fg)$ for its enveloping algebra and $S(\fg)$ for the symmetric superalgebra.
Now let the indices $i,j$ run through $1,\dots,m+n$.
Set $|i|=0$ if $1\leq i\leq m$ and $|i|=1$ if $m<i\leq m+n$.
The elements $e_{i,j}x^r:=e_{i,j}\otimes x^r$ with $r=0,1,2,\dots$ and $i,j=1,\dots,m+n$ make a basis of $\fg$.
The $\ZZ_2$-grading on $\fg$ is defined by $\deg e_{i,j}x^r=|i|+|j|$.
The supercommutation relations with $r,s\geq 0$ is given by
\begin{align}\label{Lie bracket of g}
[e_{i,j}x^r,e_{k,l}x^s]=\delta_{k,j}e_{i,l}x^{r+s}-(-1)^{(|i|+|j|)(|k|+|l)|}\delta_{l,i}e_{k,j}x^{r+s}.
\end{align}

The adjoint action of $\fg$ on itself extends uniquely to actions of $\fg$ on $U(\fg)$ and
$S(\fg)$ by derivations.
The corresponding invariant subalgebras are denoted $U(\fg)^{\fg}$ and $S(\fg)^{\fg}$.
In particular, the center $Z(\fg)=Z(\fg)_{\bar{0}}\oplus Z(\fg)_{\bar{1}}=U(\fg)^{\fg}$,
where
$$Z(\fg)_{i}=\{z\in U(\fg)_i;~za=(-1)^{ij}az,\forall~a\in(\fg)_j~\text{for}~j\in\ZZ_2\},~~i\in\ZZ_2.$$
There is one obvious family of even central elements in $U(\fg)$.
For any $r\in\NN$, we set
\begin{align}\label{center elements in g sigma}
z_r:=e_{1,1}x^r+\cdots+e_{m+n,m+n}x^r\in\fg.
\end{align}
Using (\ref{Lie bracket of g}) one can show by direct computation
that the set $\{z_r;~r\geq 0\}$ forms a basis for the center $\fz(\fg)$ of $\fg$,
so that $\kk[z_r;~r\geq 0]$ is a subalgebra of $Z(\fg)$.

\subsection{Symmetric invariants}
There is a filtration
\begin{align}
U(\fg)=\bigcup\limits_{r\geq 0}{\rm F}_rU(\fg)
\end{align}
of the enveloping algebra $U(\fg)$, which is defined by placing $e_{i,j}x^r$ in degree $r+1$,
i.e., ${\rm F}_rU(\fg)$ is the span of all monomials of the form $e_{i_1,j_1}x^{r_1}\cdots e_{i_s,j_s}x^{r_s}$ with total degree $(r_1+1)+\cdots+(r_s+1)\leq r$.
The associated graded algebra ${\rm gr}  U(\fg)$ is isomorphic (both as a graded algebra and as a graded $\fg$-module) to
$S(\fg)$.
It follows that
\begin{align}\label{graded Z(g_sigma) subset S(g)^g}
{\rm gr}  Z(\fg)\subseteq S(\fg)^{\fg}.
\end{align}

\begin{Lemma}\label{lemma symmetric invariant of S(gs)^gs}
The invariant algebra $S(\fg)^{\fg}$ is generated by $\{z_r;~r\geq 0\}$ together with $((\fg)_{\bar{0}})^p:=\{a^p;~a\in(\fg)_{\bar{0}}\}\subseteq S(\fg)$.
In fact, $S(\fg)^{\fg}$ is freely generated by
\begin{align}\label{free generators in S(gs)^gs}
\{z_r;~r\geq 0\}\cup\big\{(e_{i,j}x^r)^p;~1\leq i,j\leq m+n~{\rm with}~(i,j)\neq (1,1), r\geq 0, e_{i,j}x^r\in(\fg)_{\bar{0}}\big\}.
\end{align}
\end{Lemma}
\begin{proof}
The proof is essentially the same as \cite[Lemma 3.2]{BT18},
except that we need consider the odd elements.
If $a\in(\fg)_{\bar{0}}$, then the Leibniz rule implies $a^p\in S(\fg)^{\fg}$.
Let $I(\fg)$ be the subalgebra of $S(\fg)^{\fg}$ generated by $\{z_r;~r\geq 0\}$ and $((\fg)_{\bar{0}})^p$.
Let
$$B_0:=\{(i,j,r);~1\leq i,j\leq m+n~{\rm with}~(i,j)\neq (1,1), r\geq 0, e_{i,j}x^r\in(\fg)_{\bar{0}}\},$$

$$B_1:=\{(i,j,r);~1\leq i,j\leq m+n~{\rm with}~r\geq 0, e_{i,j}x^r\in(\fg)_{\bar{1}}\},$$
and $B:=B_0\cup B_1$ for short.
Since the elements $\{z_r;~r\geq 0\}\cup\{e_{i,j}x^r;~(i,j,r)\in B_0\cup B_1\}$ give a basis of $\fg$,
it follows that
$$S(\fg)=\kk[z_r;r\geq 0][e_{i,j}x^r;~(i,j,r)\in B_0]\otimes\Lambda[e_{i,j}x^r;~(i,j,r)\in B_1],$$
$$I(\fg)=\kk[z_r;r\geq 0][(e_{i,j}x^r)^p;~(i,j,r)\in B_0],$$
where $\Lambda[e_{i,j}x^r;~(i,j,r)\in B_1]$ is the exterior algebra (cf. \cite[(1.5.1)]{CW13}).
Hence, $S(\fg)$ is free as an $I(\fg)$-module with basis $\{\prod_{(i,j,r)\in B_0\cup B_1}(e_{i,j}x^r)^{w(i,j,r)};~\omega\in\Omega\}$,
where
$$
\Omega:=
\left\{
\omega:B\rightarrow \NN;
\begin{array}{l}
0 \leq \omega(i,j,r)<p,~\forall (i,j,r)\in B_0~\text{and}~\omega(i,j,r)\in\{0,1\},~\forall (i,j,r)\in B_1\\
\text{$\omega(i,j,r)=0$ for all but finitely many }(i,j,r)\in
B
\end{array}
\right\}.
$$
Now, we must show that $S(\fg)^{\fg}\subseteq I(\fg)$.
Given $f\in S(\fg)^{\fg}$, we thus write
$$f=\sum\limits_{\omega\in\Omega}c_{\omega}\prod\limits_{(i,j,r)\in B}(e_{i,j}x^r)^{\omega(i,j,r)}$$
for $c_{\omega}\in I(\fg)$,
all but finitely many of which are zero.
Also fix a non-zero function $\omega$,
we have to prove that $c_{\omega}=0$.

Suppose first that $\omega(i,j,r)>0$ for some $(i,j,r)\in B$ with $i\neq j$.
Choose an integer $s$ that it is bigger than all $r'$ such that $\omega(i',j',r')>0$ for $(i',j',r')\in B$,
we have
\begin{multline*}
\ad(e_{i,i}x^s)(f) =
\sum_{\omega\in\Omega}
c_{\omega}
\sum_{\substack{(i',j',r')\in B \\ \omega(i',j',r')> 0}}
(-1)^{{\rm sgn}(\omega,i',j',r')}\omega(i',j',r')
(e_{i',j'} x^{r'})^{\omega(i',j',r')-1}
\left[e_{i,i}x^s, e_{i',j'}x^{r'}\right]\\\times
\prod_{\substack{(i'',j'',r'') \in B \\ (i'',j'',r'')\neq (i',j',r')}}
(e_{i'',j''}x^{r''})^{\omega(i'',j'',r'')},
\end{multline*}
where ${\rm sgn}(\omega,i',j',r')\in\{0,1\}$ depends on $\omega,i',j',r'$.
Thanks to the choice of $s$,
the coefficient of
$$
(e_{i, j}x^{r})^{\omega(i,j,r)-1}e_{i, j}x^{s+r}
\prod_{\substack{(i'',j'',r'')\in B \\ (i'',j'',r'')\neq (i,j,r)}}
  (e_{i'',j''}x^{r''})^{\omega(i'',j'',r'')}
$$
in this expression is $(-1)^{{\rm sgn}(\omega,i,j,r)}c_\omega\omega(i,j,r)$.
It must be zero since $f\in S(\fg)^{\fg}$.
As $\omega(i,j,r)$ is non-zero in $\kk$,
we conclude that $c_\omega = 0$ as required.

By the same token, we can treat the case that $\omega(j,j,r)>0$ for some $(j,j,r)\in B$.
\end{proof}

\subsection{Restricted Lie superalgebra}
A Lie superalgebra $\mathfrak{l}=\mathfrak{l}_{\bar{0}}\oplus\mathfrak{l}_{\bar{1}}$ is called a {\it restricted Lie superalgebra}
if $(\fl_{\bar{0}},[p])$ is a restricted Lie algebra and $\fl_{\bar{1}}$ is a restricted $\fl_{\bar{0}}$-module.
By definition, for each $x\in\fl_{\bar{0}}$, the element $x^p-x^{[p]}\in U(\fl)$ is central and the map $\xi:\fl\rightarrow U(\fl);~x\mapsto x^p-x^{[p]}$ is $p$-semilinear.

For any associative $\kk$-superalgebra $A=A_{\bar{0}}\oplus A_{\bar{1}}$,
there is a natural way to define a Lie bracket $[,]$ in $A$, i.e., by the equality,
\begin{align}\label{associative super is Lie super}
[a,b]:=ab-(-1)^{|a||b|}ba.
\end{align}
The Lie superalgebra $(A, [,])$ will be denoted $A^{-}$.
Since we are in characteristic $p>0$,
the mapping $a\rightarrow a^p;~a\in A_{\bar{0}}$ endows $A^-$ with the restricted structure.

\begin{Lemma}\label{lemma g and g_sigma are restricted}
The current superalgebra $\fg$ is a restricted Lie superalgebra with $p$-map defined on the basis by the rule $(ax^r)^{[p]}:=a^{[p]}x^{rp}$,
where $a^{[p]}$ denotes the $p\mathrm{th}$ matrix power of $a\in\mathfrak{gl}_m\oplus\mathfrak{gl}_n$.
\end{Lemma}
\begin{proof}
Let ${\rm Mat}_{m|n}$ be the matrix superalgebra.
By definition, we have $\mathfrak{gl}_{m|n}=({\rm Mat}_{m|n})^{-}$,
so that $\mathfrak{gl}_{m|n}$ is restricted with the $p$-map given by the $p$th power of matrices.
Then the claim follows immediately from the rules of Lie bracket (\ref{Lie bracket of g}) and $p$-map.
\end{proof}

\subsection{The center of $U(\fg)$}
We refer to $Z_p(\fg):=\kk\langle x^p-x^{[p]};~x\in(\fg)_{\bar{0}}\rangle$ as the {\it$p$-center} of $U(\fg)$.
Since the $p$-map is $p$-semilinear, we have that
\begin{align}\label{p-center is free polynomial}
Z_p(\fg)=\kk\left[\big(e_{i,j}x^r\big)^p-\delta_{i,j}e_{i,j}x^{rp};~1\leq i,j\leq m+n, r\geq 0, |i|+|j|=0\right]
\end{align}
as a free polynomial algebra.

\begin{Theorem}\label{center of enveloping algebra of shifted current algebra}
The center $Z(\fg)$ of $U(\fg)$ is generated by $\{z_r;~r\geq 0\}$ and $Z_p(\fg)$.
In fact, $Z(\fg)$ is freely generated by
\begin{align}\label{free generator in Z(Ugs)}
\{z_r;~r\geq 0\}\cup\big\{(e_{i,j}x^r)^p-\delta_{i,j}e_{i,j}x^{rp};~(i,j)\neq (1,1), r\geq 0, |i|+|j|=0\big\}.
\end{align}
\end{Theorem}
\begin{proof}
The proof is similar to proof of \cite[Theorem 3.4]{BT18}, and will be skipped here.
\end{proof}

\begin{Remark}
In view of Theorem \ref{center of enveloping algebra of shifted current algebra},
the center $Z(\fg)$ consists of only even elements, i.e., $Z(\fg)=Z(\fg)_{\bar{0}}$.
This fact can be proved directly by the triangular decomposition of basic classical Lie
superalgebras (see \cite[(2.2.2)]{CW13}).
\end{Remark}

\section{Modular super Yangian $\Ymn$ and Drinfeld-type presentation}\label{section: modular superYangians}
In this section, we study the Yangian $Y_{m|n}$ in positive characteristic.

\subsection{\boldmath RTT presentation of $Y_{m|n}$}
Following \cite{Na91}, the super Yangian associated to the general linear Lie superalgebra $\mathfrak{gl}_{m|n}$,
denoted by $Y_{m|n}$, is the associated superalgebra over $\kk$ with the {\it RTT generators} $\{t_{i,j}^{(r)};~1\leq i,j\leq m+n, r\geq 1\}$ subject to the following relations:
\begin{align}\label{RTT relations}
\left[t_{i,j}^{(r)}, t_{k,l}^{(s)}\right] =(-1)^{|i||j|+|i||k|+|j||k|}\sum_{t=0}^{\min(r,s)-1}
\left(t_{k, j}^{(t)} t_{i,l}^{(r+s-1-t)}-
t_{k,j}^{(r+s-1-t)}t_{i,l}^{(t)}\right),
\end{align}
where the parity of $t_{i,j}^{(r)}$ is defined by $|i|+|j|~(\text{mod}~2)$,
and the bracket is understood as the supercommutator.
By convention, we set $t_{i,j}^{(0)}:=\delta_{i,j}$.

The element $t_{i,j}^{(r)};~r>0$
is called an {\it even} ({\it odd}, respectively)
element if its parity is $0$ ($1$, respectively).
We define the formal power series
\begin{align*}
t_{i,j}(u):= \sum_{r \geq 0}t_{i,j}^{(r)}u^{-r} \in Y_{m|n}[[u^{-1}]],
\end{align*}
and a matrix $T(u):=\big(t_{i,j}(u)\big)_{1\leq i,j\leq m+n}$.
It is easily seen that, in terms of the generating series,
the initial defining relation (\ref{RTT relations}) may be rewritten as follows:
\begin{align}\label{tiju tkl relation}
[t_{i,j}(u),t_{k,l}(v)]=\frac{(-1)^{|i||j|+|i||k|+|j||k|}}{(u-v)}(t_{k,j}(u)t_{i,l}(v)-t_{k,j}(v)t_{i,l}(v)).
\end{align}
Note that the matrix $T(u)$ is invertible,
we observe the following notation for the entries of the inverse of the matrix $T(u)$:
$$T(u)^{-1}=:\left(t_{i,j}'(u)\right)_{i,j=1}^{m+n},$$
and we also have another relation (see \cite[(5)]{Gow07}, \cite[(2.4)]{Peng11}):
\begin{align}\label{tiju tkl' relation}
[t_{i,j}(u),t'_{k,l}(v)]=\frac{(-1)^{|i||j|+|i||k|+|j||k|}}{(u-v)}(\delta_{k,j}\sum\limits_{s=1}^{m+n}t_{i,s}(u)t'_{s,l}(v)-\delta_{i,l}\sum\limits_{s=1}^{m+n}t'_{k,s}(u)t_{s,j}(v)).
\end{align}

For homogeneous elements $x_1,\dots,x_s$ in a superalgebra $A$,
a {\it supermonomial} in $x_1,\dots,x_s$ means a monomial of the form $x_1^{i_1}\dots x_s^{i_s}$ for some $i_1,\dots,j_s\in\ZZ_{>0}$ and
$i_j\leq 1$ if $x_j$ is odd.
The following proposition is a {\it PBW theorem} for $Y_{m|n}$,
where the proof in \cite[Theorem 1]{Gow07} works perfectly in positive characteristic.

\begin{Theorem}\label{PBW theorem Ymn}
Ordered supermonomial in the generators $\{t_{i,j}^{(r)};~1\leq i,j\leq m+n, r\geq 1\}$ taken in some fixed order forms a linear basis for $Y_{m|n}$.
\end{Theorem}

\subsection{Loop filtration}
Define the {\it loop filtration} on $Y_{m|n}$
\begin{align}\label{loop filtration}
Y_{m|n}=\bigcup_{r\geq 0}{\rm F}_rY_{m|n}
\end{align}
by setting ${\rm deg}t_{i,j}^{(r)}=r-1$,
i.e., ${\rm F}_rY_{m|n}$ is the span of all supermonomials of the form $t_{i_1,j_1}^{(r_1)}\dots t_{i_m,j_m}^{(r_m)}$ with $(r_1-1)+\cdots+(r_m-1)\leq r$.
To describe the associated graded superalgebra ${\rm gr} Y_{m|n}$,
We recall that $U(\fg)$ has the natural filtration and grading with ${\rm deg}e_{i,j}x^r=r$.

\begin{Lemma}\label{lemma U(g) isomorphic to graded of loop filtration}
The assignment
$$t_{i,j}^{(r)}\mapsto (-1)^{|i|}e_{i,j}x^{r-1}$$
gives rise to an isomorphism ${\rm gr} Y_{m|n}\cong U(\fg)$ of graded superalgebras.
\end{Lemma}
\begin{proof}
Relation (\ref{RTT relations}) implies that
\begin{align*}
[\gr_{r} t_{i,j}^{(r)}, \gr_{s} t_{k,l}^{(s)}] &= [t_{i,j}^{(r)}, t_{k,l}^{(s)}] +
{\rm F}_{r+s-3}Y_{m|n}\\
&=(-1)^{|i||j|+|i||k|+|j||k|}(\delta_{k,j}t_{i,l}^{(r+s-1)}-\delta_{i,l}t_{k,j}^{(r+s-1)})+{\rm F}_{r+s-3}Y_{m|n}\\
&=(-1)^{|i||j|+|i||k|+|j||k|}(\delta_{k,j}\gr_{r+s-2}t_{i,l}^{(r+s-1)}-\delta_{i,l}\gr_{r+s-2}t_{k,j}^{(r+s-1)}).
\end{align*}
Comparing with (\ref{Lie bracket of g}),
we deduce that the map in the statement of the
lemma is well defined.
To see that it is an isomorphism,
one uses the PBW basis from Theorem \ref{PBW theorem Ymn} to see that a basis for $\gr Y_{m|n}$
is sent to a basis for $U(\fg)$.
\end{proof}

\subsection{\boldmath Gauss decomposition and quasideterminants}\label{subsection Gauss decomp}
Note that the leading minors of the matrix $T(u)$ are always invertible and hence the matrix $T(u)$ possesses a Gauss decomposition
\begin{align}\label{gauss decomp}
T(u)=F(u)D(u)E(u)
\end{align}
for unique matrices
$$
D(u) = \left(
\begin{array}{cccc}
d_{1}(u) & 0&\cdots&0\\
0 & d_{2}(u) &\cdots&0\\
\vdots&\vdots&\ddots&\vdots\\
0&0 &\cdots&d_{m+n}(u)
\end{array}
\right),
$$$$
E(u) =
\left(
\begin{array}{cccc}
1 & e_{1,2}(u) &\cdots&e_{1,m+n}(u)\\
0 & 1 &\cdots&e_{2,m+n}(u)\\
\vdots&\vdots&\ddots&\vdots\\
0&0 &\cdots&1
\end{array}
\right),\:
F(u) = \left(
\begin{array}{cccc}
1 & 0 &\cdots&0\\
f_{2,1}(u) & 1 &\cdots&0\\
\vdots&\vdots&\ddots&\vdots\\
f_{m+n,1}(u)&f_{m+n,2}(u) &\cdots&1
\end{array}
\right).
$$
In terms of quasideterminants of \cite{GR97},
we have the following descriptions (cf. \cite[Section 3]{Gow07}):
\begin{align}\label{quasideterminants D}
d_i(u) =
\left|
\begin{array}{cccc}
t_{1,1}(u) & \cdots & t_{1,i-1}(u)&t_{1,i}(u)\\
\vdots & \ddots &\vdots&\vdots\\
t_{i-1,1}(u)&\cdots&t_{i-1,i-1}(u)&t_{i-1,i}(u)\\
t_{i,1}(u) & \cdots & t_{i,i-1}(u)&
\hbox{\begin{tabular}{|c|}\hline$t_{i,i}(u)$\\\hline\end{tabular}}
\end{array}
\right|,
\end{align}
\begin{equation}\label{quasideterminants E}
e_{i,j}(u) =
d_i(u)^{-1} \left|
\begin{array}{cccc}
t_{1,1}(u) & \cdots &t_{1,i-1}(u)& t_{1,j}(u)\\
\vdots & \ddots &\vdots&\vdots\\
t_{i-1,1}(u) & \cdots &t_{i-1,i-1}(u)&t_{i-1,j}(u)\\
t_{i,1}(u) & \cdots & t_{i,i-1}(u)&
\hbox{\begin{tabular}{|c|}\hline$t_{i,j}(u)$\\\hline\end{tabular}}
\end{array}
\right|,
\end{equation}
\begin{equation}\label{quasideterminants F}
f_{j,i}(u) =
\left|
\begin{array}{cccc}
t_{1,1}(u) & \cdots &t_{1,i-1}(u)& t_{1,i}(u)\\
\vdots & \ddots &\vdots&\vdots\\
t_{i-1,1}(u) & \cdots & t_{i-1,i-1}(u)&t_{i-1,i}(u)\\
t_{j,1}(u) & \cdots & t_{j,i-1}(u)&
\hbox{\begin{tabular}{|c|}\hline$t_{j,i}(u)$\\\hline\end{tabular}}
\end{array}
\right|{d}_i(u)^{-1}.
\end{equation}
We use the following notation for the coefficients:
$$d_i(u)=\sum\limits_{r\geq 0}d_i^{(r)}u^{-r};~~~(d_i(u))^{-1}=\sum\limits_{r\geq 0}d_i'^{(r)}u^{-r};$$
$$e_{i,j}(v)=\sum\limits_{r\geq 1}e_{i,j}^{(r)}v^{-r};~~~f_{j,i}(v)=\sum\limits_{r\geq 1}f_{j,i}^{(r)}v^{-r}.$$
Let $e_j(u):=e_{j,j+1}(u),~f_j(v):=f_{j+1,j}(v)$ for short.
By the above, we immediately have $e_{j-1}^{(1)}=t_{j-1,j}^{(1)}$ and $f_{j-1}^{(1)}=t_{j,j-1}^{(1)}$.
By induction, one may show that for each pair $i,j$ such that $1<i+1<j\leq m+n-1$, we have
\begin{align}\label{induction eij,fij}
e_{i,j}^{(r)}=(-1)^{|j-1|}[e_{i,j-1}^{(r)},e_{j-1}^{(1)}];~f_{j,i}^{(r)}=(-1)^{|j-1|}[f_{j-1}^{(1)},f_{j-1,i}^{(r)}].
\end{align}

By multiplying out the matrix product (\ref{gauss decomp}),
we see that each $t_{i,j}^{(r)}$ can be expressed as a sum of monomials in $d_i^{(r)}, e_{i,j}^{(r)}$ and $f_{j,i}^{(r)}$,
appearing in certain order that all $f's$ before $d's$ and all $d's$ before $e's$.
By (\ref{induction eij,fij}), it is enough to use $d_i^{(r)}, e_{j}^{(r)}$ and $f_j^{(r)}$ only, so that we have the following theorem.

\begin{Theorem}\label{Ymn generated by fde}
The super Yangian $Y_{m|n}$ is generated as an algebra by the following elements:
$$\{d_i^{(r)};~1\leq i\leq m+n, r\geq 0\}\cup\{e_j^{(r)},f_j^{(r)};~1\leq j\leq m+n-1,r\geq 1\}.$$
\end{Theorem}

\subsection{Maps between Super Yangians}
To explicitly write down the relations among the Drinfeld generators,
we start with the special cases when $m$ and $n$ are either $1$ or $2$, that are relatively less complicated,
and then to apply the maps in this section to obtain the relations in the general case.

\begin{Proposition}\label{proposition auto}
\begin{enumerate}
\item The map $\rho_{m|n}:Y_{m|n}\rightarrow Y_{n|m}$ defined by
$$\rho_{m|n}(t_{i,j}(u))=t_{m+n+1-i,m+n+1-j}(-u)$$
is an algebra isomorphism.
\item The map $\omega_{m|n}:Y_{m|n}\rightarrow Y_{m|n}$ defined by
$$\omega_{m|n}(T(u))=(T(-u))^{-1}$$
is an algebra isomorphism.
\item For any $k\in\ZZ_{\geq 0}$, the map $\psi_k: Y_{m|n}\rightarrow Y_{m+k|n}$ defined by
$$\psi_k=\omega_{m+k|n}\circ\varphi_{m|n}\circ\omega_{m|n},$$
where $\varphi_{m|n}: Y_{m|n}\rightarrow Y_{m+k|n}$ is the injective algebra homomorphism which sends each $t_{i,j}^{(r)}\in Y_{m|n}$ to $t_{k+i,k+j}^{(r)}\in Y_{m+k|n}$.
\item The map $\zeta_{m|n}:Y_{m|n}\rightarrow Y_{n|m}$ defined by
$$\zeta_{m|n}=\rho_{m|n}\circ\omega_{m|n}$$
is an algebra isomorphism.
\end{enumerate}
\end{Proposition}
\begin{proof}
The proof follows easily from the defining relations, see also \cite[Section 4]{Gow07}.
\end{proof}

We call $\psi_k$ the {\it shift map} and $\zeta_{m|n}$ the {\it swap map}.
The following Proposition is due to Gow and the proof given in characteristic zero in \cite[Proposition 1, Lemma 2]{Gow07} works as well in positive characteristic.
\begin{Proposition}\label{image of zeta and psi_k}
\begin{enumerate}
\item Let $1\leq i\leq m+n$ and $1\leq j\leq m+n-1$. We have
\begin{align*}
\zeta_{m|n}(d_i(u))=(d_{m+n+1-i}(u))^{-1},
\end{align*}
\begin{align*}
\zeta_{m|n}(e_j(u))=-f_{m+n-j}(u),
\end{align*}
\begin{align*}
\zeta_{m|n}(f_j(u))=-e_{m+n-j}(u).
\end{align*}
\item For $k,l\geq 1$, we have
\begin{align}\label{psi d}
\psi_k(d_l(u))=d_{k+l}(u),
\end{align}
\begin{align}\label{psi e}
\psi_k(e_l(u))=e_{k+l}(u),
\end{align}
\begin{align}\label{psi f}
\psi_k(f_l(u))=f_{k+l}(u).
\end{align}
\item The subalgebras $Y_k$ and $\psi_k(Y_{m|n})$ in $Y_{m+k|n}$ supercommute with each other.
\end{enumerate}
\end{Proposition}

\subsection{Gauss decomposition of $Y_{m|n}$}
We first consider the case, where $m=2$ and $n=1$.
The following lemma is a generalization and modular analogue of \cite[Lemma 3]{Gow07}.
\begin{Lemma}\label{identities in Y21}
The following identities hold in $Y_{2|1}[[u^{-1},v^{-1},w^{-1}]]$:
\begin{align}\label{identities in Y21-1}
(u-v)[d_i(u),e_j(v)]=\left\{\begin{array}{ll}
			(\delta_{ij}-\delta_{i,j+1})d_i(u)(e_j(v)-e_j(u)),&\mbox{if~}j=1;\\
			(\delta_{ij}+\delta_{i,j+1})d_i(u)(e_j(v)-e_j(u)),&\mbox{if~}j=2;
		\end{array}\right.
\end{align}
\begin{align}\label{identities in Y21-2}
(u-v)[d_i(u),f_j(v)]=\left\{\begin{array}{ll}
			-(\delta_{ij}-\delta_{i,j+1})(f_j(v)-f_j(u))d_i(u),&\mbox{if~}j=1;\\
			-(\delta_{ij}+\delta_{i,j+1})(f_j(v)-f_j(u))d_i(u),&\mbox{if~}j=2;
		\end{array}\right.
\end{align}
\begin{align}\label{identities in Y21-3}
(u-v)[e_j(u), f_k(v)]=(-1)^{|j+1|}
		\delta_{jk}\left(d_j(u)^{-1}d_{j+1}(u)-d_j(v)^{-1}d_{j+1}(v)\right);	
\end{align}
\begin{align}\label{identities in Y21-4}
(u-v)[e_j(u), e_j(v)]=(-1)^{|j+1|}\left(e_j(u)-e_j(v)\right)^2;
\end{align}
\begin{align}\label{identities in Y21-5}
(u-v)[f_j(u), f_j(v)]=-(-1)^{|j+1|}\left(f_j(u)-f_j(v)\right)^2;
\end{align}
\begin{align}\label{identities in Y21-6}
(u-v)\left[e_1(u),e_2(v)\right]=e_1(u)e_2(v)-e_1(v)e_2(v)-e_{1,3}(u)+e_{1,3}(v);
\end{align}
\begin{align}\label{identities in Y21-7}
(u-v)\left[f_1(u),f_2(v)\right]=-f_2(v)f_1(u)+f_2(v)f_1(v)+f_{3,1}(u)-f_{3,1}(v);
\end{align}
\begin{align}\label{identities in Y21-8}
\left[e_{13}(u),e_{2}(v)\right]=e_{2}(v)\left[e_{1}(u),e_2(v)\right];
\end{align}
\begin{align}\label{identities in Y21-9}
\left[e_1(u),e_{13}(v)-e_1(v)e_2(v)\right]=-\left[e_1(u),e_{2}(v)\right]e_1(u);
\end{align}
\begin{align}\label{identities in Y21-10}
\left[\left[e_i(u),e_j(v)\right],e_j(v)\right]=0\quad\mbox{if~}|i-j|=1;
\end{align}
\begin{align}\label{identities in Y21-11}
\left[\left[f_i(u),f_j(v)\right],f_j(v)\right]=0\quad\mbox{if~}|i-j|=1;
\end{align}
\begin{align}\label{identities in Y21-12}
\left[\left[e_i(u),e_j(v)\right],e_j(w)\right]+\left[\left[e_i(u),e_j(w)\right],e_j(v)\right]=0,\quad\mbox{if~}|i-j|=1;
\end{align}
\begin{align}\label{identities in Y21-13}
\left[\left[f_i(u),f_j(v)\right],f_j(w)\right]+\left[\left[f_i(u),f_j(w)\right],f_j(v)\right]=0,\quad\mbox{if~}|i-j|=1.
\end{align}
\end{Lemma}
\begin{proof}
Equations (\ref{identities in Y21-1})-(\ref{identities in Y21-3}) were proved over $\CC$ in \cite[Lemma 3]{Gow07},
the same proof works here.

To establish (\ref{identities in Y21-4}), we first consider the Yangian $Y_2$.
According to \cite[(4.30)]{BT18} (see also \cite[Lemma 5.4(iv)]{BK05}), we have $(u-v)[e_1(u),e_1(v)]=(e_1(u)-e_1(v))^2$.
The standard embedding $Y_2\hookrightarrow Y_{2|1}$ yields (\ref{identities in Y21-6}) for $j=1$.
Using exactly the same degree argument as in \cite[(5.4)]{Peng11},
we see that $(u-v)[e_1(u),e_1(v)]=-(e_1(u)-e_1(v))^2$ holds in the Yangian $Y_{1|1}$.
By applying the shift map $\psi_1:Y_{1|1}\rightarrow Y_{2|1}$ (see (\ref{psi e})),
we obtain (\ref{identities in Y21-4}) for $j=2$.
The relation (\ref{identities in Y21-5}) follows from a consecutive application of \cite[(5.6)]{BK05} and the swap map $\zeta_{1|1}$ to (\ref{identities in Y21-4}) with suitable choices of indices.

For the remaining relations,
we just give a brief account.
This can be proven in the same manner as \cite[Section 6]{Peng11}.
For example, for (\ref{identities in Y21-12}), it suffices to verify that
$$(u-v)(u-w)(v-w)[[e_1(u),e_2(v)],e_2(w)]$$
is symmetric in $v$ and $w$.
Actually, this follows from (\ref{identities in Y21-6})-(\ref{identities in Y21-11}).
To establish (\ref{identities in Y21-6}),
we use (\ref{tiju tkl' relation}) to get
$$(u-v)[t_{1,2}(u),t'_{2,3}(v)]=t_{1,1}(u)t'_{1,3}(v)+t_{1,2}(u)t'_{2,3}(v)+t_{1,3}(u)t'_{3,3}(v).$$
Substituting by the Drinfeld's generators (see for example \cite[Page 807]{Gow07}) in the above identity,
we have
\begin{align*}
[u-v][d_1(u)e_1(u),-e_2(v)d_3(v)^{-1}]
=&d_1(u)(e_1(v)e_2(v)-e_{1,3}(v)
-e_1(u)e_2(v)+e_{1,3}(u))d_3(v)^{-1}.
\end{align*}
Then we multiply both sides on the left by $d_{1}(u)^{-1}$ and on the right by $d_3(v)$ to get (\ref{identities in Y21-6}).
For (\ref{identities in Y21-8}), we have by (\ref{tiju tkl' relation}) and (\ref{identities in Y21-1}) that
$$0=[t_{1,3}(u),-t'_{2,3}(v)]=[d_1(u)e_{1,3}(u),e_2(v)d_3(v)^{-1}]=d_1(u)[e_{1,3}(u),e_2(v)d_3(v)^{-1}].$$
It follows that $[e_{1,3}(u),e_2(v)d_3(v)^{-1}]=0$, this readily yields
$$(\ast)\ \ [e_{1,3}(u),e_2(v)]d_3(v)^{-1}=e_2(v)[e_{1,3}(u),d_3(v)^{-1}].$$
Note that (\ref{identities in Y21-1}) gives the following identity:
$$(u-v)[d_3(v)^{-1},e_2(u)]=(e_2(u)-e_2(v))d_3(v)^{-1}.$$
Taking the coefficient of $u^0$ and using the Super-Jacobi identity, we have
$$[e_{1,3}(u),d_3(v)^{-1}]=[e_1{u},e_2(v)]d_3(v)^{-1}.$$
Now substituting this into $(\ast)$ gives (\ref{identities in Y21-8}).
The proof of (\ref{identities in Y21-9}) is similar to (\ref{identities in Y21-8}),
and we need to consider the relation $[t_{1,2}(u),t'_{1,3}(v)]=0$ using (\ref{tiju tkl' relation}).
For (\ref{identities in Y21-10}) in the case $i=1,j=2$, we have by (\ref{identities in Y21-6}) and (\ref{identities in Y21-8}) to see that
$$(u-v+1)[[e_1(u),e_2(v)],e_2(v)]=[[e_1(v),e_2(v)],e_2(v)].$$
Now let $u=v-1$ to deduce that the right-hand side is zero, then divide by $(u-v+1)$ to complete the proof.
\end{proof}

For general $m,n$, the relations among $d's, e's$ and $f's$ are given in the following lemma,
which is a generalization and modular analogue of \cite[Lemma 4]{Gow07}.
\begin{Lemma}\label{identities in Ymn}
The following relations hold in $Y_{m|n}[[u^{-1},v^{-1},w^{-1}]]$:
\begin{align}\label{identities in Ymn-1}
[d_i(u),d_j(v)]=0 \quad\mbox{for~all}~1\leq i,j\leq m+n;
\end{align}
\begin{align}\label{identities in Ymn-2}
[e_i(u),e_j(v)]=0=[f_i(u),f_j(v)]\quad\mbox{if~}|i-j|>1;
\end{align}
\begin{align}\label{identities in Ymn-3}
(u-v)[d_i(u),e_j(v)]=(-1)^{|i|}((\delta_{ij}-\delta_{i,j+1})d_i(u)(e_j(v)-e_j(u)));
\end{align}
\begin{align}\label{identities in Ymn-4}
(u-v)[d_i(u),f_j(v)]=(-1)^{|i|}(-\delta_{ij}+\delta_{i,j+1})(f_j(v)-f_j(u))d_i(u);
\end{align}
\begin{align}\label{identities in Ymn-5}
(u-v)[e_i(u), f_j(v)]=(-1)^{|j+1|}
		\delta_{ij}\left(d_i(u)^{-1}d_{i+1}(u)-d_i(v)^{-1}d_{i+1}(v)\right);	
\end{align}
\begin{align}\label{identities in Ymn-6}
(u-v)[e_j(u), e_j(v)]=(-1)^{|j+1|}\left(e_j(u)-e_j(v)\right)^2;
\end{align}
\begin{align}\label{identities in Ymn-7}
(u-v)[f_j(u), f_j(v)]=-(-1)^{|j+1|}\left(f_j(u)-f_j(v)\right)^2;
\end{align}
\begin{align}\label{identities in Ymn-8}
(u-v)[e_j(u), e_{j+1}(v)]=(-1)^{|j+1|}\left(e_j(u)e_{j+1}(v)-e_j(v)e_{j+1}(v)
-e_{j,j+2}(u)+e_{j,j+2}(v)\right);
\end{align}
\begin{align}\label{identities in Ymn-9}
(u-v)[f_j(u), f_{j+1}(v)]=-(-1)^{|j+1|}\left(f_{j+1}(v)f_{j}(u)-f_{j+1}(v)f_{j}(v)
-f_{j+2,j}(u)+f_{j+2,j}(v)\right);
\end{align}
\begin{align}\label{identities in Ymn-10}
\left[\left[e_i(u),e_j(v)\right],e_j(v)\right]=0\quad\mbox{if~}|i-j|=1;
\end{align}
\begin{align}\label{identities in Ymn-11}
\left[\left[f_i(u),f_j(v)\right],f_j(v)\right]=0\quad\mbox{if~}|i-j|=1;
\end{align}
\begin{align}\label{identities in Ymn-12}
[[e_i(u),e_j(v)],e_j(w)]+[[e_i(u),e_j(w)],e_j(v)]=0,\quad\mbox{if~}|i-j|=1;
\end{align}
\begin{align}\label{identities in Ymn-13}
[[f_i(u),f_j(v)],f_j(w)]+[[f_i(u),f_j(w)],f_j(v)]=0,\quad\mbox{if~}|i-j|=1;
\end{align}
\begin{align}\label{identities in Ymn-14}
[[e_{i-1}^{(r)},e_i^{(1)}],[e_i^{(1)},e_{i+1}^{(s)}]]=0;
\end{align}
\begin{align}\label{identities in Ymn-15}
[[f_{i-1}^{(r)},f_i^{(1)}],[f_i^{(1)},f_{i+1}^{(s)}]]=0.
\end{align}
\end{Lemma}
\begin{proof}
The relations (\ref{identities in Ymn-1}) and (\ref{identities in Ymn-2}) follow directly from Proposition \ref{image of zeta and psi_k}(3).

For the relations (\ref{identities in Ymn-3})-(\ref{identities in Ymn-13}), we use Lemma \ref{identities in Y21}, \cite[Theorem 4.3]{BT18},
together with the shift maps and swap maps.
We just go through (\ref{identities in Ymn-10}) in the case $j=i+1$, i.e.,
$$\left[\left[e_i(u),e_{i+1}(v)\right],e_{i+1}(v)\right]=0,$$
since the others are similar.
There are four cases:
(a) $1\leq i\leq m-2$;
(b) $i=m-1$;
(c) $i=m$;
(d) $m+1\leq i\leq m+n-2$.
For these, (a) and (d) are immediate from \cite[(4.25)]{BT18}.
We consider the composite map:
$$\eta:Y_{2|1}\stackrel{\psi_{m-2}}{\longrightarrow}Y_{m|1}\stackrel{\zeta_{m|1}}{\longrightarrow}Y_{1|m}\stackrel{\psi_{n-1}}{\longrightarrow}Y_{n|m}\stackrel{\zeta_{n|m}}{\longrightarrow}Y_{m|n}.$$
Now Proposition \ref{image of zeta and psi_k} implies $\eta(e_1)=e_{m-1}$ and $\eta(e_2)=e_{m}$, so that (\ref{identities in Y21-10}) yields (b).
By using (\ref{identities in Y21-11}) instead of (\ref{identities in Y21-10}), (c) is a similar argument to (b).

The proof of (\ref{identities in Ymn-15}) is similar to (\ref{identities in Ymn-14}),
so we just prove (\ref{identities in Ymn-14}).
Taking its coefficient of $u^{-r}v^{-s}w^{-t}$ in (\ref{identities in Ymn-12}), we have
\begin{align}\label{ymn-12 coefficients}
\left[[e_i^{(r)},e_j^{(s)}],e_j^{(t)}\right]+\left[[e_i^{(r)},e_j^{(t)}],e_j^{(s)}\right]=0,\quad\mbox{if~}|i-j|=1.
\end{align}
Then taking the $u^{-r}v^{-2t}$-coefficient in (\ref{identities in Ymn-10}) in conjunction with (\ref{ymn-12 coefficients}) gives
\begin{align}\label{ymn-10 coefficients}
\left[[e_i^{(r)},e_j^{(t)}],e_j^{(t)}\right]=0\quad\mbox{if~}|i-j|=1.
\end{align}
To show (\ref{identities in Ymn-14}), we consider the following three cases:
(i) $1\leq i\leq m-1$;
(ii) $i=m$;
(iii) $m+1\leq i\leq m+n-2$.
As before, (iii) is a similar argument to (i), so we just prove (i) and (ii).
Suppose that $1\leq i\leq m-1$. By (\ref{ymn-10 coefficients}), (\ref{identities in Ymn-2}) and super-Jacobi identity, we have:
\begin{eqnarray*}
[[e_{i-1}^{(r)},e_i^{(1)}],[e_i^{(1)},e_{i+1}^{(s)}]]&=&[e_{i}^{(1)},[[e_{i-1}^{(r)},e_i^{(1)}],e_{i+1}^{(s)}]]=[e_{i}^{(1)},[e_{i+1}^{(s)},[e_i^{(1)},e_{i-1}^{(r)}]]\\
&=&[e_{i}^{(1)},[[e_{i+1}^{(s)},e_i^{(1)}],e_{i-1}^{(r)}]]=[[e_{i+1}^{(s)},e_i^{(1)}],[e_{i}^{(1)},e_{i-1}^{(r)}]]\\
&=&-[[e_{i-1}^{(r)},e_i^{(1)}],[e_i^{(1)},e_{i+1}^{(s)}]].
\end{eqnarray*}
Hence we have proved (i) in case $\Char\kk=p\neq 2$.
The proof of (ii) is the same argument as in the proof of \cite[Lemma 5]{Gow07}.
Now we assume that $p=2$ and consider the composite map
$$\xi:Y_{2|2}\stackrel{\psi_{i-2}}{\longrightarrow}Y_{i|2}\stackrel{\zeta_{i|2}}{\longrightarrow}Y_{2|i}\stackrel{\psi_{m+n-i-2}}{\longrightarrow}Y_{m+n-i|i}\stackrel{\zeta_{m+n-i|i}}{\longrightarrow}Y_{i|m+n-i}.$$
Note that $Y_{i|m+n-i}=Y_{m|n}=Y_{m+n}$.
Then Gow's argument (\cite[Lemma 5]{Gow07}) in conjunction with the map $\xi$ yields (i) for $p=2$.
\end{proof}
\begin{Remark}\label{rk2}
(1) Note that (\ref{identities in Ymn-6}) and (\ref{identities in Ymn-7}) are different from the ones in \cite[(30)-(31)]{Gow07}.
Suppose that $p\neq 2$.
We may switch $u$ and $v$.
Since both $e_m(u)$ and $e_m(v)$ are odd, this forces $(u-v)[e_m(u),e_m(v)]=0$.
If $p=2$, then $Y_{m|n}\cong Y_{m+n}$, so that $(u-v)[e_m(u),e_m(v)]=-(e_m(u)-e_m(v))^2=(e_m(u)-e_m(v))^2$ (cf. \cite[(4.30)]{BT18}).

(2) For the {\it quartic Serre relations} (\ref{identities in Ymn-14}, \ref{identities in Ymn-15}),
we need consider all admissible $i$.
If $p\neq 2$, then the proof of Lemma \ref{identities in Ymn} implies that the quartic Serre relations
for $j\neq m$ already follow from the quadratic and cubic relations,
however we cannot derive the quartic Serre relations from others in the case $p=2$ (cf. \cite[Theorem 4.3]{BT18}).
\end{Remark}

\subsection{Drinfeld-type presentation}
Recall that the fact in Section \ref{subsection Gauss decomp} that $Y_{m|n}$ is generated as an algebra by the set $\{d_i^{(r)},d_i'^{(r)},e_j^{(r)},f_j^{(r)}\}$.
The following theorem describes the relations among these generators, and
the relations  are enough as defining relations of the super Yangian $Y_{m|n}$ (cf. \cite[Theorem 3]{Gow07}).

\begin{Theorem}\label{theorem Drinfeld presentation}
The Yangian $Y_{m|n}$ is generated by the elements $\{d_i^{(r)},d_i'^{(r)};~1\leq i\leq m+n, r\geq 1\}$ and $\{e_j^{(r)},f_j^{(r)};~1\leq j\leq m+n-1, r\geq 1\}$ subject only to the following relations:
\begin{align}\label{di di' relation}
d_i^{(0)}=1,~\sum\limits_{t=0}^rd_i^{(t)}d_i'^{(r-t)}=\delta_{r0};
\end{align}
\begin{align}\label{di dj commu}
[d_i^{(r)},d_j^{(s)}]=0;
\end{align}
\begin{align}\label{Drinfeld generators relation 1}
[d_i^{(r)},e_j^{(s)}]=(-1)^{|i|}(\delta_{ij}-\delta_{i,j+1})\sum\limits_{t=0}^{r-1}d_i^{(t)}e_j^{(r+s-1-t)};
\end{align}
\begin{align}\label{Drinfeld generators relation 2}
[d_i^{(r)},f_j^{(s)}]=-(-1)^{|i|}(\delta_{ij}-\delta_{i,j+1})\sum\limits_{t=0}^{r-1}f_j^{(r+s-1-t)}d_i^{(t)};
\end{align}
\begin{align}\label{Drinfeld generators relation 3}
[e_i^{(r)},f_j^{(s)}]=-(-1)^{|i+1|}\delta_{ij}\sum\limits_{t=0}^{r+s-1}d_i'^{(t)}d_{i+1}^{(r+s-1-t)};
\end{align}
\begin{align}\label{Drinfeld generators relation 4}
[e_j^{(r)},e_j^{(s)}]=(-1)^{|j+1|}(\sum\limits_{t=1}^{s-1}e_j^{(t)}e_j^{r+s-1-t}-\sum\limits_{t=1}^{r-1}e_j^{(t)}e_j^{(r+s-1-t)});
\end{align}
\begin{align}\label{Drinfeld generators relation 5}
[f_j^{(r)},f_j^{(s)}]=(-1)^{|j+1|}(\sum\limits_{t=1}^{r-1}f_j^{(t)}f_j^{r+s-1-t}-\sum\limits_{t=1}^{s-1}f_j^{(t)}f_j^{(r+s-1-t)});
\end{align}
\begin{align}\label{Drinfeld generators relation 6}
[e_j^{(r+1)},e_{j+1}^{(s)}]-[e_j^{(r)},e_{j+1}^{(s+1)}]=(-1)^{|j+1|}e_j^{(r)}e_{j+1}^{(s)};
\end{align}
\begin{align}\label{Drinfeld generators relation 7}
[f_j^{(r+1)},f_{j+1}^{(s)}]-[f_j^{(r)},f_{j+1}^{(s+1)}]=-(-1)^{|j+1|}f_{j+1}^{(s)}f_{j}^{(r)};
\end{align}
\begin{align}\label{Drinfeld generators relation 0}
[e_i^{(r)},e_j^{(s)}]=0=[f_i^{(r)},f_j^{(s)}]\quad\mbox{if~}|i-j|>1;
\end{align}
\begin{align}\label{Drinfeld generators relation 8}
[[e_i^{(r)},e_j^{(s)}],e_j^{(t)}]+[[e_i^{(r)},e_j^{(t)}],e_j^{(s)}]=0,\quad\mbox{if~}|i-j|=1;
\end{align}
\begin{align}\label{Drinfeld generators relation 9}
[[f_i^{(r)},f_j^{(s)}],f_j^{(t)}]+[[f_i^{(r)},f_j^{(t)}],f_j^{(s)}]=0,\quad\mbox{if~}|i-j|=1;
\end{align}
\begin{align}\label{Drinfeld generators relation 10}
[[e_i^{(r)},e_j^{(t)}],e_j^{(t)}]=0\quad\mbox{if~}|i-j|=1;
\end{align}
\begin{align}\label{Drinfeld generators relation 11}
[[f_i^{(r)},f_j^{(t)}],f_j^{(t)}]=0,\quad\mbox{if~}|i-j|=1;
\end{align}
\begin{align}\label{Drinfeld generators relation 12}
[[e_{i-1}^{(r)},e_i^{(1)}],[e_i^{(1)},e_{i+1}^{(s)}]]=0;
\end{align}
\begin{align}\label{Drinfeld generators relation 13}
[[f_{i-1}^{(r)},f_i^{(1)}],[f_i^{(1)},f_{i+1}^{(s)}]]=0.
\end{align}
\end{Theorem}
\begin{Remark}\label{rk4}
(1) There are several differences between Theorem \ref{theorem Drinfeld presentation} and \cite[Theorem 3]{Gow07}.
The relations (41)-(43) of {\it loc. cit.} are expressed here as the two relations (\ref{Drinfeld generators relation 4})-(\ref{Drinfeld generators relation 5}).
Also relations (47)-(48) of {\it loc. cit.} are expressed here as the four relations (\ref{Drinfeld generators relation 8})-(\ref{Drinfeld generators relation 11}).
This is essential in characteristic $2$ (see Remark \ref{rk2}).

(2) We note that the relations (39), (44)-(45) in \cite[Theorem 3]{Gow07} contain some typos.

(3) Relations (\ref{Drinfeld generators relation 4}) and (\ref{Drinfeld generators relation 5}) are equivalent to the following relations:
\begin{align*}
[e_j^{(r+1)},e_j^{(s)}]-[e_j^{(r)},e_j^{(s+1)}]=-(-1)^{|j+1|}(e_j^{(s)}e_j^{(r)}+e_j^{(r)}e_j^{(s)}),
\end{align*}
\begin{align*}
[f_j^{(r+1)},f_j^{(s)}]-[f_j^{(r)},f_j^{(s+1)}]=(-1)^{|j+1|}(f_j^{(s)}f_j^{(r)}+f_j^{(r)}f_j^{(s)}).
\end{align*}
\end{Remark}

\begin{proof}
We have before proved (\ref{Drinfeld generators relation 10})-(\ref{Drinfeld generators relation 12}) (see (\ref{identities in Ymn-14})-(\ref{identities in Ymn-15}) and (\ref{ymn-10 coefficients})) and (\ref{Drinfeld generators relation 13}) is a similar argument to (\ref{Drinfeld generators relation 12}),
while the others come from
Lemma \ref{identities in Ymn} and the identity
\begin{align}\label{gv-gu/u-v}
\frac{g(v)-g(u)}{u-v}=\sum\limits_{r,s\geq 1}g^{(r+s-1)}u^{-r}v^{-s}
\end{align}
for any formal series $g(u)=\sum_{r\geq 0}g^{(r)}u^{-r}$.

Now we consider the second part of the proof.
Let $\widehat{Y}_{m|n}$ be the algebra generated by elements and relations as in the theorem.
We may further define all the other $e_{i,j}^{(r)}$ and $f_{j,i}^{(r)}$ in $\widehat{Y}_{m|n}$ by setting $e_{i,i+1}^{(r)}:=e_i^{(r)}$ and $f_{j+1,j}^{(r)}:=f_j^{(r)}$,
then using the formula (\ref{induction eij,fij}) inductively when $|i-j|>1$.
Let $\theta:\widehat{Y}_{m|n}\rightarrow Y_{m|n}$ be the map sending every element in $\widehat{Y}_{m|n}$ into the element in $Y_{m|n}$ with the same name.
The previous paragraph implies that $\theta$ is a well-defined surjective homomorphism.
Therefore, it remains to prove that $\theta$ is also injective.
The injectivity will be proved in Proposition \ref{theta is injective}.
\end{proof}

Let $\widehat{Y}^+_{m|n}$ denote the subalgebra of $\widehat{Y}_{m|n}$ generated by the elements $\{e_{i,j}^{(r)};~1\leq i<j\leq m+n,r>0\}$.
Define a filtration on $\widehat{Y}^+_{m|n}$ by declaring that the elements $e_{i,j}^{(r)}$ are of filtered degree $r-1$,
and denote by ${\rm gr} \widehat{Y}^+_{m|n}$ the corresponding graded algebra.
Let $\bar{e}_{i,j}^{(r)}:=\gr_{r-1}e_{i,j}^{(r)}$ be the image of $e_{i,j}^{(r)}$ in the graded algebra ${\rm gr}_{r-1}\widehat{Y}^+_{m|n}$.

The following equations were proven over $\CC$ in \cite[(53)-(55)]{Gow07},
Using (\ref{Drinfeld generators relation 6}), (\ref{Drinfeld generators relation 8}) and (\ref{induction eij,fij}),
this can be obtained by exactly the same proof.
\begin{Lemma}\label{pre lemma}
The following identities hold in ${\rm gr}\widehat{Y}^+_{m|n}$:
\begin{align}
[\bar{e}_{i,i+1}^{(r)},\bar{e}_{k,k+1}^{(s)}]&=0,~{\rm if}~|i-k|\neq 1\label{pre 1},\\
[\bar{e}_{i,i+1}^{(r+1)},\bar{e}_{k,k+1}^{(s)}]&=[\bar{e}_{i,i+1}^{(r)},\bar{e}_{k,k+1}^{(s+1)}],~{\rm if}~|i-k|\neq 1\label{pre 2},\\
[\bar{e}_{i,i+1}^{(r)},[\bar{e}_{i,i+1}^{(s)},\bar{e}_{k,k+1}^{(t)}]]&=-[\bar{e}_{i,i+1}^{(s)},[\bar{e}_{i,i+1}^{(r)},\bar{e}_{k,k+1}^{(t)}]]~{\rm if}~|i-k|\neq 1\label{pre 3},\\
\bar{e}_{i,j}^{(r)}=(-1)^{|j-1|}[\bar{e}_{i,j-1}^{(r)},\bar{e}_{j-1,j}^{(1)}]&=(-1)^{|i+1|}[\bar{e}_{i,i+1}^{(1)},\bar{e}_{i+1,j}^{(r)}],~{\rm for}~j>i+1\label{pre 4}.
\end{align}
\end{Lemma}

\begin{Proposition}\label{theta is injective}
The map $\theta$ in Theorem \ref{theorem Drinfeld presentation} is injective.
\end{Proposition}
\begin{proof}
By the same argument as in \cite[pp. 814]{Gow07} (see also \cite[Section 8]{Peng11}), in order to proving the injectivity of $\theta$,
it suffices to show that the following relation
\begin{align}\label{relation in graded algebra}
[\bar{e}_{i,j}^{(r)},\bar{e}_{k,l}^{(s)}]=(-1)^{|j|}\delta_{kj}\bar{e}_{il}^{(r+s-1)}-(-1)^{|i||j|+|i||k|+|j||k|}\delta_{li}\bar{e}_{kj}^{(r+s-1)}.
\end{align}
The proof of (\ref{relation in graded algebra}) is the same argument as in the proof of \cite[Theorem 3]{Gow07} except the following four relations (cf. \cite[(56)-(59)]{Gow07}):
\begin{align}
[\bar{e}_{i,i+2}^{(r)},\bar{e}_{i+1,i+2}^{(s)}]&=0,~{\rm for}~1\leq i\leq m+n-2\label{diff 1},\\
[\bar{e}_{i,i+1}^{(r)},\bar{e}_{i,i+2}^{(s)}]&=0,~{\rm for}~1\leq i\leq m+n-2\label{diff 2},\\
[\bar{e}_{i,i+2}^{(r)},\bar{e}_{i+1,i+3}^{(s)}]&=0,~{\rm for}~1\leq i\leq m+n-3\label{diff 3},\\
[\bar{e}_{i,j}^{(r)},\bar{e}_{k,k+1}^{(s)}]&=0,~{\rm for}~1\leq i\leq k<j\leq m+n\label{diff 4}.
\end{align}
By Lemma \ref{pre lemma},
we have
\begin{eqnarray*}
(-1)^{|i+1|}[\bar{e}_{i,i+2}^{(r)},\bar{e}_{i+1,i+2}^{(s)}]&\overset{(\ref{pre 4})}{=}&[[\bar{e}_{i,i+1}^{(r)},\bar{e}_{i+1,i+2}^{(1)}],\bar{e}_{i+1,i+2}^{(s)}]\\
&\overset{(\ref{pre 3})}{=}&-[[\bar{e}_{i,i+1}^{(r)},\bar{e}_{i+1,i+2}^{(s)}],\bar{e}_{i+1,i+2}^{(1)}]\\
&\overset{(\ref{pre 2})}{=}&-[[\bar{e}_{i,i+1}^{(r+s-1)},\bar{e}_{i+1,i+2}^{(1)}],\bar{e}_{i+1,i+2}^{(1)}],
\end{eqnarray*}
and the last term is zero by (\ref{Drinfeld generators relation 10}), this proves (\ref{diff 1}).
For (\ref{diff 2}), the same method in (\ref{diff 1}) works, except that we apply (\ref{pre 4}) on the term $\bar{e}_{i,i+2}^{(s)}$.
By applying (\ref{pre 4}) on the left side of (\ref{diff 3}), we obtain
\begin{eqnarray*}
[\bar{e}_{i,i+2}^{(r)},\bar{e}_{i+1,i+3}^{(s)}]=(-1)^{|i+1|+|i+2|}[[\bar{e}_{i,i+1}^{(r)},\bar{e}_{i+1,i+2}^{(1)}],[\bar{e}_{i+1,i+2}^{(1)},\bar{e}_{i+2,i+3}^{(s)}]],
\end{eqnarray*}
which is zero by (\ref{Drinfeld generators relation 12}), hence (\ref{diff 3}) is true.

To establish (\ref{diff 4}), we first consider the case $i=k$.
We proceed by induction on $j-i$.
When $j-i=1$ (resp. $j-i=2$), this follows from (\ref{pre 1}) (resp. (\ref{diff 2})).
The super-Jacobi identity in conjunction with (\ref{pre 4}) gives
\begin{eqnarray*}
(-1)^{|j-1|}[\bar{e}_{i,i+1}^{(s)},\bar{e}_{i,j}^{(r)}]&=&[\bar{e}_{i,i+1}^{(s)},[\bar{e}_{i,j-1}^{(r)},\bar{e}_{j-1,j}^{(1)}]]\\
&=&[[\bar{e}_{i,i+1}^{(s)},\bar{e}_{i,j-1}^{(r)}],\bar{e}_{j-1,j}^{(1)}]+(-1)^{(|i|+|i+1|)(|i|+|j-1|)}[\bar{e}_{i,j-1}^{(r)},[\bar{e}_{i,i+1}^{(s)},\bar{e}_{j-1,j}^{(1)}]]\\
&=&0,
\end{eqnarray*}
where the last equality follows from the induction hypothesis and (\ref{pre 1}).
For the case $i<k$ in (\ref{diff 4}), we use (\ref{pre 4}) to reduce the problem to showing
$$[\bar{e}_{i,k+1}^{(r)},\bar{e}_{k,k+1}^{(s)}]=0=[\bar{e}_{i,k+1}^{(r)},\bar{e}_{k,k+2}^{(s)}].$$
By using again (\ref{pre 4})-(\ref{diff 3}),
this follows from the induction on $k-i$.
\end{proof}

Recall that by Lemma \ref{lemma U(g) isomorphic to graded of loop filtration},
we may identify $e_{i,j}x^{r}$ with $(-1)^{|i|}{\rm gr}_rt_{i,j}^{(r+1)}$ via the identification
$U(\fg)$ and ${\rm gr} Y_{m|n}$.
Using (\ref{quasideterminants D})-(\ref{quasideterminants F}),
one sees that $d_i^{(r+1)}, e_{i,j}^{(r+1)}$ and $f_{j,i}^{(r+1)}$ all belong to ${\rm F}_rY_{m|n}$.
Combining with \cite[Proposition 8.1, 8.4]{Peng11}, and under our identification we have that
\begin{equation}\label{identification}
e_{i,j}x^{r}
=\left\{
\begin{array}{ll}
(-1)^{|i|}\gr_{r}d_i^{(r+1)}&\text{if $i=j$,}\\
(-1)^{|i|}\gr_{r}e_{i,j}^{(r+1)}&\text{if $i<j$,}\\
(-1)^{|i|}\gr_{r}f_{i,j}^{(r+1)}&\text{if $i>j$.}
\end{array}\right.
\end{equation}

Using PBW theorem for $U(\fg)$, we obtian
the PBW basis for $Y_{m|n}$ (cf. \cite[Corollary 8.5]{Peng11}).

\begin{Theorem}\label{PBW for Ymn}
Ordered supermonomials
in the elements
\begin{equation*}
\{d_i^{(r)};~1\leq i \leq m+n, r > 0\} \cup
\{e_{i,j}^{(r)}, f_{j,i}^{(r)}\:\big|\: 1\leq i < j \leq m+n, r >
0\}
\end{equation*}
taken in any fixed ordering
form a basis for $Y_{m|n}$.
\end{Theorem}

\begin{Lemma}\label{useful lemma 1}
The following relations hold in $Y_{m|n}[[u^{-1}, v^{-1}]]$
for all $l \geq 0$:
\begin{align}
(u-v)[e_i(u),(e_i(v)-e_i(u))^l]&=(-1)^{|i+1|}l(e_i(v)-e_i(u))^{l+1}\nonumber\\
&=(-1)^{|i|}l(e_i(v)-e_i(u))^{l+1},\label{new1}\\
(u-v)[e_i(u), d_i(v)(e_i(v)-e_i(u))^l] &=(-1)^{|i|}(l-1)d_i(v)(e_i(v)-e_i(u))^{l+1},\label{new2}\\
(u-v)[e_i(u), d_{i+1}(v)(e_i(v)-e_i(u))^l] &=(-1)^{|i+1|}(l+1)d_{i+1}(v)(e_i(v)-e_i(u))^{l+1},\label{new3}\\
(u-v)[e_i(u), d_{i+1}(v)(e_i(v)\!-\!e_i(u))^ld_i(v)^{-1}] &=(-1)^{|i+1|}(l+2)d_{i+1}(v)(e_i(v)\!-\!e_i(u))^{l+1}d_i(v)^{-1}.\label{new4}
\end{align}
\end{Lemma}
\begin{proof}
We first check (\ref{new1}).
If $i\neq m$, then $|i|=|i+1|$ and $e_i(u)$ is even.
In this case, the relation (\ref{new1}) follows from (\ref{identities in Ymn-6}) and the Leibniz rule.
It remains to treat the case $i=m$.
There is nothing to prove, if $p=2$. So assume $p\neq 2$.
Now Remark \ref{rk2} readily yields
$$(u-v)[e_m(u),e_m(v)]=-(e_m(v)-e_m(u))^2=0,$$
so that all terms of (\ref{new1}) are equal to $0$.

In view of (\ref{identities in Ymn-3}), we have
\begin{align}
(u-v)[e_i(u),d_i(v)]&=(-1)^{|i|}d_i(v)(e_i(u)-e_i(v)),\label{new5}\\
(u-v)[e_i(u),d_i(v)^{-1}]&=(-1)^{|i|}(e_i(v)-e_i(u))d_i(v)^{-1},\label{new6}\\
(u-v)[e_i(u),d_{i+1}(v)]&=(-1)^{|i+1|}d_{i+1}(v)(e_i(v)-e_i(u)),\label{new7}\\
(u-v)[e_i(u),d_{i+1}(v)^{-1}]&=(-1)^{|i+1|}(e_i(u)-e_i(v))d_{i+1}(v)^{-1}.\label{new8}
\end{align}
Then (\ref{new2})-(\ref{new4}) follows from (\ref{new1}), (\ref{new5})-(\ref{new7}) using Leibniz again.
\end{proof}

The following result follows from (\ref{new5}) and (\ref{new7}) by specializing $v$.
\begin{Corollary}\label{ed}
The following hold in $Y_{m|n}[[u^{-1}]]$:
\begin{align}
e_i(u-(-1)^{|i|})d_i(u)&=d_i(u)e_i(u),&
d_i(u)^{-1}
e_i(u-(-1)^{|i|})&=e_i(u)d_i(u)^{-1},\label{ed1}\\
e_i(u+(-1)^{|i+1|})d_{i+1}(u)&=d_{i+1}(u)e_i(u), &
d_{i+1}(u)^{-1}e_i(u+(-1)^{|i+1|})&=e_i(u)d_{i+1}(u)^{-1}.\label{ed 2}
\end{align}
\end{Corollary}

\begin{Lemma}\label{useful lemma 2}
For any $i = 1,\dots,m+n-1$, $l \geq 0$ and $r,s > 0$, we have that
\begin{align}\label{111}
\Bigg[e_i^{(r)}, \hspace{-6mm}\sum_{\substack{s_1,\dots,s_l \geq r \\ s_1+\cdots+s_l=
    (l-1)(r-1)+s}}
\hspace{-10mm}
e_i^{(s_1)}
\cdots
e_i^{(s_l)}\Bigg]
&=(-1)^{|i|}l
\hspace{-6mm}\sum_{\substack{s_1,\dots,s_{l+1} \geq r \\ s_1+\cdots+s_{l+1} =
    l(r-1)+s}}
\hspace{-10mm}e_i^{(s_1)}
\cdots
e_i^{(s_{l+1})},\\
\Bigg[e_i^{(r)}, \hspace{-6mm}\sum_{\substack{s_1,\dots,s_l\leq r-1 \\ s_1+\cdots+s_l =
    (l-1)(r-1)+s}}\hspace{-10mm}
e_i^{(s_1)}
\cdots
e_i^{(s_l)}\Bigg]
&= -(-1)^{|i|}l\hspace{-6mm}
\sum_{\substack{s_1,\dots,s_{l+1} \leq r-1 \\ s_1+\cdots+s_{l+1} =
    l(r-1)+s}}\hspace{-10mm}
e_i^{(s_1)}
\cdots
e_i^{(s_{l+1})},\label{222}\\
\Bigg[e_i^{(r)}, \hspace{-6mm}\sum_{\substack{s_1,\dots,s_l \geq r, t \geq 0 \\ s_1+\cdots+s_l+t =
    l(r-1)+s}}\hspace{-10mm}
d_i^{(t)}e_i^{(s_1)}
\cdots
e_i^{(s_l)}\Bigg]
&= (-1)^{|i|}(l-1)\hspace{-12mm}
\sum_{\substack{s_1,\dots,s_{m+1} \geq r, t \geq 0 \\ s_1+\cdots+s_{m+1}+t =
    (l+1)(r-1)+s}}
\hspace{-14mm}d_i^{(t)}e_i^{(s_1)}
\cdots
e_i^{(s_{l+1})},\label{333}
\end{align}\begin{align}
\Bigg[e_i^{(r)}, \hspace{-6mm}\sum_{\substack{s_1,\dots,s_m \geq r, t \geq 0 \\ s_1+\cdots+s_l+t =
    l(r-1)+s}}\hspace{-10mm}
d_{i+1}^{(t)}e_i^{(s_1)}
\cdots
e_i^{(s_l)}\Bigg]
&=(-1)^{|i+1|}(l+1)
\hspace{-12mm}\sum_{\substack{s_1,\dots,s_{l+1} \geq r, t \geq 0 \\ s_1+\cdots+s_{l+1}+t =
    (l+1)(r-1)+s}}
\hspace{-14mm}d_{i+1}^{(t)}e_i^{(s_1)}
\cdots
e_i^{(s_{l+1})},\label{444}\\
\Bigg[e_i^{(r)}, \hspace{-8mm}\sum_{\substack{s_1,\dots,s_l \geq r,
    t\geq 0, u\geq 0 \\
    s_1+\cdots+s_l+t+u =
    l(r-1)+s}}\hspace{-12mm}
d_{i+1}^{(t)}e_i^{(s_1)}
\cdots
e_i^{(s_l)}
d_i'^{(u)}\Bigg]
&=(-1)^{|i+1|}(l+2)\hspace{-16mm}
\sum_{\substack{s_1,\dots,s_{l+1} \geq r, t\geq 0, u \geq 0 \\
    s_1+\cdots+s_{l+1} +t+u=
    (l+1)(r-1)+s}}
\hspace{-17mm}d_{i+1}^{(t)} e_i^{(s_1)}
\cdots
e_i^{(s_{l+1})}
d_i'^{(u)}.\label{555}
\end{align}
\end{Lemma}
\begin{proof}
Using (\ref{Drinfeld generators relation 4}) and the same argument as in the proof of Lemma \ref{useful lemma 1} one obtains
\begin{align}\label{666}
[e_i^{(r)},e_i^{(s_j)}]=(-1)^{|i|}\sum_{\substack{s_j', s_j''\geq r\\ s_j'+s_j''=s_j+r-1}}
\!\!\!\!\!e_i^{(s_j')}e_i^{(s_j'')}=(-1)^{|i+1|}\sum_{\substack{s_j', s_j''\geq r\\ s_j'+s_j''=s_j+r-1}}
\!\!\!\!\!e_i^{(s_j')}e_i^{(s_j'')}
\end{align}
for $0<r\leq s_j$, and (\ref{new6}) in conjunction with (\ref{gv-gu/u-v}) gives
\begin{align}\label{777}
[e_i^{(r)},d_i'^{(s)}] =(-1)^{|i|}\sum_{t=0}^{s-1}e_i^{(r+s-1-t)}d_i'^{(t)}.
\end{align}
If $p=2$ or $i\neq m$, then the lemma can be proved in the same method as in \cite[Lemma 4.9]{BT18} using the relations (\ref{666}) and (\ref{777}),
and hence we skip the detail.

So we assume $i=m$ and $p>2$.
For (\ref{111}), we shall show
$$(\dag)\ \ \ \ [e_i^{(r)}, \hspace{-6mm}\sum_{\substack{s_1,\dots,s_l \geq r \\ s_1+\cdots+s_l=
    (l-1)(r-1)+s}}
\hspace{-10mm}
e_i^{(s_1)}
\cdots
e_i^{(s_l)}]=0=(-1)^{|i|}l
\hspace{-6mm}\sum_{\substack{s_1,\dots,s_{l+1} \geq r \\ s_1+\cdots+s_{l+1} =
    l(r-1)+s}}
\hspace{-10mm}e_i^{(s_1)}
\cdots
e_i^{(s_{l+1})}.$$
It is clear that the above equality holds if $l=0$.
Since $e_m^{(r)}$ is odd, (\ref{Drinfeld generators relation 4}) implies that $[e_m^{(r)},e_m^{(s)}]=0$ for all $r,s>0$,
so that in particular $(e_m^{(r)})^2=0$.
This yields
$$\sum_{\substack{s_1,\dots,s_l \geq r \\ s_1+\cdots+s_l=
    (l-1)(r-1)+s}}
\hspace{-10mm}
e_i^{(s_1)}
\cdots
e_i^{(s_l)}=0$$
whenever $l\geq 2$. Hence, $(\dag)$ holds for all $l\geq 1$.
The other parts are proved similarly.
\end{proof}

We define
\begin{align*}
d_{i\downarrow k}(u)&:=d_i(u)d_i(u-1)\cdots d_i(u-k+1),\\
d_{i\uparrow k}(u)&:=d_i(u)d_i(u+1)\cdots d_i(u+k-1).
\end{align*}

\begin{Lemma}\label{cyclic product of di}
The following relations hold for all $k\geq 1$:
\begin{align}
(u-v)[d_{i\downarrow k}(u), e_i(v)]&=kd_{i\downarrow k}(u)(e_i(v)-e_i(u)),~{\rm for}~1\leq i\leq m,\label{d arrow e 1}\\
(u-v)[d_{i\uparrow k}(u), e_i(v)] &=-kd_{i\uparrow k}(u)(e_i(v)-e_i(u)),~{\rm for}~m+1\leq i\leq m+n-1,\label{d arrow e 2}\\
(u-v)[d_{i\uparrow k}(u),e_{i-1}(v)]&=kd_{i\uparrow k}(u)(e_{i-1}(u)-e_{i-1}(v)),~{\rm for}~2\leq i\leq m,\label{d arrow e 3}\\
(u-v)[d_{i\downarrow k}(u),e_{i-1}(v)]&=-kd_{i\downarrow k}(u)(e_{i-1}(u)-e_{i-1}(v)),~{\rm for}~m+1\leq i\leq m+n.\label{d arrow e 4}
\end{align}
\end{Lemma}
\begin{proof}
We only prove (\ref{d arrow e 2}), as others can be treated similarly.
Actually, we will prove it in the following equivalent form:
\[ (\ast)\ \ \
(u-v+k)d_{i\uparrow k}(u)e_i(v)
= (u-v)e_i(v)
d_{i\uparrow k}(u)
+kd_{i\uparrow k}(u)e_i(u).
\]
This follows when $k=1$ from (\ref{new5}).
To prove $(\ast)$ in general,
we proceed by induction on $k$.
Given $(\ast)$ for some $k \geq 1$, multiply both sides on the left by
$(u-v+k+1)d_i(u+k)$ to deduce that:
\begin{align}\label{sssss}
(u-v+k+1)(u-v+k)d_{i\uparrow k+1}(u)e_i(v)&=(u-v)(u-v+k+1)d_i(u+k)e_i(v)
d_{i\uparrow k}(u)\nonumber \\
&+k(u-v+k+1)d_{i\uparrow k+1}(u)e_i(u).
\end{align}
Using the case of $k=1$ in $(\ast)$ and replacing $u$ by $u+k$ give that
$$(u-v+k+1)d_{i}(u+k)e_i(v)
= (u+k-v)e_i(v)
d_{i}(u+k)
+d_{i}(u+k)e_i(u+k).$$
Then substituting the above identity into (\ref{sssss}) and using (\ref{ed1}) we obtain $(\ast)$ with $k$ replaced by $k+1$,
as required.
\end{proof}

We shall consider more diagonal elements,
we let
\begin{equation}\label{definition h}
h_i(u)=\sum_{r \geq 0}h_i^{(r)}u^{-r}:=-(-1)^{|i|}d_{i+1}(u)d_i(u)^{-1}
\end{equation}
assuming $1\leq i\leq m+n-1$.
According to (\ref{di di' relation}), we have
$d_i'^{(0)}=1~$ and $d_i'^{(r)}=-\sum_{t=1}^rd_i^{(t)}d_i'^{(r-t)}$,
so that in particular $h_i^{(r+1)}\in {\rm F}_rY_{m|n}, h_i^{(0)}=-(-1)^{|i|}$ and $\gr_rd_i^{(r+1)}=-\gr_rd_i'^{(r+1)}$.
Moreover, the identification (\ref{identification}) yields:
\begin{equation}\label{grading of hi r+1}
\gr_r h_i^{(r+1)}=e_{i,i}x^r-(-1)^{|i|+|i+1|}e_{i+1,i+1}x^{r}.
\end{equation}
Note also by Corollary \ref{ed} that
\begin{equation}\label{h e comm relation}
h_i(u)e_i(u-(-1)^{|i|})=e_i(u+(-1)^{|i+1|})h_i(u).
\end{equation}

\begin{Lemma}\label{h e bracket relation}
The following relations hold in $Y_{m|n}[[u^{-1}, v^{-1}]]$:
\begin{align}
(u-v-(-1)^{|i|})[h_{i}(u),e_i(v)]&=(-1)^{|i+1|}2h_{i}(u)(e_i(u-(-1)^{|i|})-e_i(v)),\label{h e bracket relation 1}\\
(u-v+(-1)^{|i+1|})[h_{i}(u),e_i(v)]&=(-1)^{|i+1|}2(e_i(u+(-1)^{|i+1|})-e_i(v))h_{i}(u),\label{h e bracket relation 2}\\
(u-v)[h_{i-1}(u),e_i(v)] &=(-1)^{|i|}h_{i-1}(u)(e_i(v)-e_i(u)),\label{h e bracket relation 3}\\
(u-v-(-1)^{|i|})[h_{i-1}(u),e_i(v)] &=(-1)^{|i|}(e_i(v)-e_i(u-(-1)^{|i|}))h_{i-1}(u),\label{h e bracket relation 4}\\
(u-v)[h_{i+1}(u),e_i(v)] &=(-1)^{|i+1|}(e_i(v)-e_i(u))h_{i+1}(u),\label{h e bracket relation 5}\\
(u-v+(-1)^{|i+1|})[h_{i+1}(u),e_i(v)] &=(-1)^{|i+1|}h_{i+1}(u)(e_i(v)-e_i(u+(-1)^{|i+1|}))\label{h e bracket relation 6}.
\end{align}
\end{Lemma}
\begin{proof}
We prove (\ref{h e bracket relation 2}), (\ref{h e bracket relation 4}) and (\ref{h e bracket relation 5}) in detail here, while the others can be proved in a similar fashion.
To establish (\ref{h e bracket relation 2}), we have by (\ref{new4}) that
\begin{equation}\label{lemma zhong xuyao}
(u-v)[e_i(u), d_{i+1}(v)d_i(v)^{-1}]=(-1)^{|i+1|}2
d_{i+1}(v)(e_i(v)-e_i(u))d_i(v)^{-1}.
\end{equation}
From (\ref{new7}), we have that
$$(v-u+(-1)^{|i+1|})d_{i+1}(v)e_i(u)=(v-u)e_i(u)d_{i+1}(v)+(-1)^{|i+1|}d_{i+1}(v)e_i(v).$$
Then we multiply (\ref{lemma zhong xuyao}) by $(v-u+(-1)^{|i+1|})$ and use these identities
in combination with (\ref{ed 2}) to obtain
$$
(u-v)(v-u+(-1)^{|i+1|})[e_i(u),h_i(v)]=(-1)^{|i+1|}2(v-u)(e_i(v+(-1)^{|i+1|})-e_i(u))h_i(v).
$$
Dividing through by $(u-v)$ and interchanging $u$ and $v$ give the result.
For (\ref{h e bracket relation 4}), the Leibniz rule for $\ad e_i(v)$ in conjunction with (\ref{identities in Ymn-3}) implies that
$$(u-v-(-1)^{|i|})[d_i(u)d_{i-1}(u)^{-1},e_i(v)]=((v-u+(-1)^{|i|}))[e_i(v),d_i(u)]d_{i-1}(u)^{-1}.$$
Substituting the last bracket by (\ref{new5}) and simplifying the result, we have
$$(u-v-(-1)^{|i|})[d_i(u)d_{i-1}(u)^{-1},e_i(v)]=(-1)^{|i|}(e_i(v)d_i(u)-d_i(u)e_i(u))d_{i-1}(u)^{-1}.$$
Then the assertion follows from applying (\ref{ed1}) to the above equality.
Finally, for (\ref{h e bracket relation 5}),
this follows easily from (\ref{new8}) and the Leibniz rule, using that $d_{i+2}(u)$ commutes with $e_i(v)$ by (\ref{identities in Ymn-3}).
\end{proof}

\begin{Corollary}\label{corollary: h e bracket}
The following hold in $Y_{m|n}[[u^{-1}, v^{-1}]]$:
\begin{align}\label{coro 1}
(u-v)[h_i(u),e_i(v)]=\left\{\begin{array}{ll}
(-1)^{|i+1|}(2h_i(u)e_i(u-(-1)^{|i|})-h_i(u)e_i(v)-e_i(v)h_i(u)),&i\neq m;\\
0,&i=m;\\
\end{array}\right.
\end{align}
\begin{align}\label{coro2}
\hspace{-15mm}(u-v)\Big[h_{i-1}\big(u+\halfa\big),e_i(v)\Big]&=
\halfa \Big(h_{i-1}(u+\halfa)e_i(v)
+e_i(v)h_{i-1}\big(u+\halfa\big)\Big)\notag\\
&\hspace{-55mm}- \halfa \Big(h_{i-1}\big(u+\halfa\big)e_i\big(u+\halfa\big)+e_i\big(u-\halfa\big)h_{i-1}\big(u+\halfa\big)\Big),
\end{align}
\begin{align}\label{coro3}
\hspace{-3mm}(u-v)\Big[h_{i+1}\big(u-\halfb\big),e_i(v)\Big] &=
\halfb\Big(h_{i+1}(u-\halfb)e_i(v)
+e_i(v)h_{i+1}\big(u-\halfb\big)\Big)\notag\\
&\hspace{-55mm}- \halfb\Big(h_{i+1}\big(u-\halfb\big)e_i\big(u+\halfb\big)+ e_i\big(u-\halfb\big)h_{i+1}\big(u-\halfb\big)\Big),
\end{align}
assuming that $\Char\kk \neq 2$ for the last two.
\end{Corollary}
\begin{proof}
Suppose that $\Char\kk\neq 2$.
The relations (\ref{coro2}) and (\ref{coro3}) follow by averaging
the corresponding pairs identities from Lemma~\ref{h e bracket relation}, e.g.
(\ref{coro2}) is
$($(\ref{h e bracket relation 3})+(\ref{h e bracket relation 4})$)/2$ with $u$ replaced by $u+\halfa$.
For (\ref{coro 1}), we note that $(-1)^{|i|}=(-1)^{|i+1|}$ when $i\neq m$.
When combined with (\ref{h e comm relation}), the first equality follows by
averaging (\ref{h e bracket relation 1}) and (\ref{h e bracket relation 2}).
If $i=m$, then the right hand side of (\ref{h e bracket relation 1}) and (\ref{h e bracket relation 2}) are equal.
In conjunction with (\ref{h e comm relation}) this implies $[h_i(u),e_i(v)]=0$.
To establish (\ref{coro 1}) when $\Char\kk=2$,
we observe by (\ref{h e bracket relation 1}) that $[h_i(u),e_i(v)]=0$,
which easily implies the desired identity.
\end{proof}

\subsection{Automorphisms}\label{Section Automorphisms}
We list the following (anti)automorphisms of $Y_{m|n}$ which are needed in the next section;
see \cite[Proposition 1.12]{MNO96} and \cite[(5.22)]{Peng21}.
\begin{enumerate}
\item(``Multiplication by a power series'')
For any power series $f(u) \in 1+u^{-1} \kk[[u^{-1}]]$, there
is an automorphism $\mu_f:Y_{m|n}\rightarrow Y_{m|n}$
defined from $\mu_f(T(u))=f(u)T(u)$,
On Drinfeld generators, it is easy to show by induction and (\ref{gauss decomp}) that
$\mu_f(d_i(u))=f(u)d_i(u)$,
$\mu_f(e_{j}(u))=e_j(u)$ and $\mu_f(f_j(u))=f_j(u)$.
\item(``Transposition'')
By the presentation of $Y_{m|n}$, there is an anti-automorphism $\tau:Y_{m|n}\rightarrow Y_{m|n}$ of order $2$ defined by
$\tau(d_i^{(r)})=d_i^{(r)}, \tau(e_j^{(r)})=f_j^{(r)},
\tau(f_j^{(r)})=e_j^{(r)}$.
\item(``Permutation'')
Let $S_{m+n}$ be the Symmetric group on $m+n$ objects and let $S_m\times S_n\subseteq S_{m+n}$ denote its Young subgroup associated to $(m,n)$.
Suppose that $p\neq 2$.
For each $w\in S_m\times S_n\subseteq S_{m+n}$, there is an automorphism
$w:Y_{m|n}\rightarrow Y_{m|n}$ sending $t_{i,j}^{(r)} \mapsto t_{w(i),
  w(j)}^{(r)}$. This is clear from the RTT relation (\ref{RTT relations}) (see also \cite[Section 2]{Peng14} and \cite[Lemma 2.24]{Tsy20}).
If $p=2$, then each element $w\in S_{m+n}$ gives rise to an automorphism (\cite[Section 4.5(5)]{BT18}).
\item(``Translation'')
For $c\in\kk$, there is an automorphism $\eta_c:\Ymn\rightarrow\Ymn$ defined from
$\eta_c(t_{i,j}(u))=t_{i,j}(u-c)$, i.e. $\eta_c(t_{i,j}^{(r)}) =
\sum_{s=1}^{r} \binom{r-1}{r-s} c^{r-s} t_{i,j}^{(s)}$.
In terms of Drinfeld generators, $\eta_c$ sends $d_i(u) \mapsto d_i(u-c), e_{i,j}(u) \mapsto
e_{i,j}(u-c)$ and $f_{j,i}(u) \mapsto f_{j,i}(u-c)$.
This can be easily checked by the relations (\ref{tiju tkl relation}) and (\ref{quasideterminants D})-(\ref{quasideterminants F}).
\end{enumerate}
\begin{Lemma}\cite[Lemma 4.16]{BT18}\label{permutation auto}
Suppose that $p>2$.
For $1\leq i<j\leq m+n$ and $|i|+|j|=0$,
the permutation automorphism of $Y_{m|n}$ defined by the transposition $(i+1,j)$ maps $e_i(u)\mapsto e_{i,j}(u)$ and $f_i(u)\mapsto f_{j,i}(u)$.
\end{Lemma}
\begin{proof}
This follows from (\ref{quasideterminants D})-(\ref{quasideterminants F}).
\end{proof}

\section{The center of $Y_{m|n}$}\label{Section:center}
In this section, we will describe the center of the modular super Yangian $\Ymn$,
and give precise formulas for the generators.

\subsection{Harish-Chandra center}
Following \cite{Na91}, we define the {\it quantum Berezinian} of the matrix $T(u)$ as the following power series:
\begin{eqnarray*}
c(u)&:=&\sum_{\rho\in S_m}\sgn(\rho)t_{\rho(1),1}(u)t_{\rho(2),2}(u-1)\cdots t_{\rho(m),m}(u-m+1)\nonumber\\
&\times &\sum_{\sigma\in S_n}\sgn(\sigma)t'_{m+1,m+\sigma(1)}(u-m+1)\cdots t'_{m+n,m+\sigma(n)}(u-m+n).
\end{eqnarray*}
Thanks to \cite[Theorem 1]{Gow05},
we may also write the quantum Berezinian as follows:
\begin{eqnarray}\label{definition c}
c(u)&=&d_1(u)d_2(u-1)\cdots d_m(u-m+1)\nonumber\times d_{m+1}(u-m+1)^{-1}\cdots d_{m+n}(u-m+n)^{-1}\nonumber\\
&=:&1+\sum\limits_{r\geq 1}c^{(r)}u^{-r}.
\end{eqnarray}
The algebra generated by the coefficients $\{c^{(r)};~r\geq 1\}$ will be denoted by $Z_\HC(Y_{m|n})$.
We call it the {\it Harish-Chandra center} of $Y_{m|n}$.
\begin{Proposition}\label{HC center of Ymn}
The elements $\{c^{(r)};~r\geq 1\}$ are central.
Furthermore, we have that $c^{(r)}$ has degree $r-1$ with respect to the loop filtration and ${\rm gr}_{r-1}c^{(r)}=z_{r-1}\in U(\fg)$.
Hence, $c^{(1)}, c^{(2)},\dots$ are algebraically independent.
\end{Proposition}
\begin{proof}
Using (\ref{di dj commu}) and the anti-automorphism $\tau$,
we just need to check $[c(u),e_i(v)]=0$ for all $i$.
This can be proven in the same manner as \cite[Theorem 7.2]{BK05} and \cite[Theorem 2]{Gow05} using the relation (\ref{identities in Ymn-3}).
By the proof of \cite[Theorem 4]{Gow07}, we have
$$c^{(r)}=t_{1,1}^{(r)}+\cdots+t_{m,m}^{(r)}-t_{m+1,m+1}^{(r)}-\cdots-t_{m+n,m+n}^{(r)}+{\rm terms~of~lower~degree}.$$
Consequently, ${\rm gr}_{r-1}c^{(r)}=z_{r-1}$ (see Lemma \ref{lemma U(g) isomorphic to graded of loop filtration}).
The final assertion follows because $z_0,z_1,\dots $ are algebraically independent in $U(\fg)$.
\end{proof}

\subsection{\boldmath Off-diagonal $p$-central elements}\label{subsection: off-diagonal p-central elements}
This subsection is a super generalization of \cite[Section 5.2]{BT18}.
We may assume $p>2$ because the case $p=2$ has been considered in {\it loc. cit}.
In this subsection,
we investigate the $p$-central elements that lie in the ``even root subalgebras''
$Y_{i,j}^{+}, Y_{j,i}^{-}\subseteq Y_{m|n}$ for $1\leq i<j\leq m+n$ and $|i|+|j|=0$,
that is, the subalgebras generated by $\{e_{i,j}^{(r)};~r>0, |i|+|j|=0\}$ and $\{f_{j,i}^{(r)};~r>0, |i|+|j|=0\}$, respectively.

\begin{Lemma}\label{root power series central}
For $1\leq i<j\leq m+n$ and $|i|+|j|=0$,
all coefficients in the power series $(e_{i,j}(u))^p$ and $(f_{j,i}(u))^p$ belong to $Z(Y_{m|n})$.
\end{Lemma}
\begin{proof}
Using Lemma \ref{permutation auto} and the anti-automorphism $\tau$ it only needs to be proved that
the coefficients of $(e_i(u))^{p}$ are central in $Y_{m|n}$ for each $i=1,\dots,m-1,m+1,\dots,m+n-1$.

Since we are in characteristic $p$, Theorem \ref{Ymn generated by fde} implies that it is enough to show the following identities
in $Y_{m|n}[[u^{-1},v^{-1}]]$ for all admissible $j, k$:
\begin{align}
(\ad e_i(u))^p (e_j(v)) &= 0,\label{p power coff 1}\\
(\ad e_i(u))^p(d_k(v)) &=0,\label{p power coff 2}\\
(\ad e_i(u))^p(f_j(v)) &= 0.\label{p power coff 3}
\end{align}
A consecutive application of the swap map and the anti-automorphism $\tau$ implies that
we may assume $1\leq i\leq m-1$.
Note that there is a standard embedding $Y_m\hookrightarrow Y_{m|n}$.
Then \cite[Lemma 5.2]{BT18} implies that equations (\ref{p power coff 1})-(\ref{p power coff 3}) hold for all $1\leq j\leq m-1$ and $1\leq k\leq m$.
Due to (\ref{identities in Ymn-2}), (\ref{identities in Ymn-3}) and (\ref{identities in Ymn-5}) it remains to prove that $(\ad\,e_{m-1}(u))^p (e_m(v))=0$.
This follows immediately from (\ref{identities in Ymn-10}).
\end{proof}

\begin{Lemma}\label{root subalgebra p center}
For $1\leq i<j\leq m+n, |i|+|j|=0$ and $r > 0$,
we have that $(e_{i,j}^{(r)})^p, (f_{j,i}^{(r)})^p\in Z(Y_{m|n})$.
\end{Lemma}
\begin{proof}
Similar to the proof of Lemma \ref{root power series central},
it is enough to show that $(e_i^{(r)})^{p}\in Z(Y_{m|n})$ for each $1\leq i\leq m-1$.
This reduces to checking
\begin{align}
(\ad e_i^{(r)})^p(e_j^{(s)})&= 0,\label{coff p power 1}\\
(\ad e_i^{(r)})^p(d_k^{(s)})&= 0,\label{coff p power 2}\\
(\ad e_i^{(r)})^p(f_j^{(s)})&= 0.\label{coff p power 3}
\end{align}
Owing to \cite[Lemma 5.3]{BT18},
the equations (\ref{coff p power 1})-(\ref{coff p power 3}) hold for all $1\leq j\leq m-1$ and $1\leq k\leq m$.
Thanks to (\ref{Drinfeld generators relation 1}), (\ref{Drinfeld generators relation 2}) and (\ref{Drinfeld generators relation 0}),
we only have to show that $(\ad e_{m-1}^{(r)})^p(e_m^{(s)})=0$.
But this follows from (\ref{Drinfeld generators relation 10}).
\end{proof}

\begin{Remark}
Lemmas \ref{root power series central} and \ref{root subalgebra p center} can also be deduced using Lemmas \ref{useful lemma 1} and \ref{useful lemma 2},
respectively.
\end{Remark}

Also, we put
\begin{align}\label{pq}
p_{i,j}(u) &=\sum_{r\geq p}p_{i,j}^{(r)}u^{-r}:=e_{i,j}(u)^p,&
q_{j,i}(u) &=\sum_{r\geq p}q_{j,i}^{(r)}u^{-r}:=f_{j,i}(u)^p.
\end{align}

\begin{Theorem}\label{theorem of off dig center}
For $1\leq i<j\leq m+n, |i|+|j|=0$,
the algebras $Z(Y_{m|n})\cap Y_{i,j}^+$ and $Z(Y_{m|n})\cap Y_{j,i}^-$
are infinite rank polynomial algebras freely generated by the central elements $\{(e_{i,j}^{(r)})^p;~r > 0\}$ and $\{(f_{j,i}^{(r)})^p;~r > 0\}$, respectively.
We have that
$(e_{i,j}^{(r)})^p, (f_{j,i}^{(r)})^p \in{\rm F}_{rp-p}Y_{m|n}$
and
\begin{equation}\label{center 1-111}
\gr_{rp-p}
(e_{i,j}^{(r)})^p=(-1)^{|i|}(e_{i,j}x^{r-1})^p,
\qquad
\gr_{rp-p}
(f_{j,i}^{(r)})^p=(-1)^{|j|}(e_{j,i}x^{r-1}\big)^p.
\end{equation}
For $r \geq p$ we have
that
\begin{equation}\label{center 1-222}
p_{i,j}^{(r)} = \left\{
\begin{array}{ll}
(-1)^{|i|}(e_{i,j}^{(r/p)})^p+(*)&\text{if $p \mid r$,}\\
(*)&\text{if $p \nmid r$,}
\end{array}
\right.
\end{equation}
where $(*)\in{\rm F}_{r-p-1} Y_n$
is a polynomial in
the elements
$(e_{i,j}^{(s)})^p$ for $1 \leq s<\lfloor r/p\rfloor$.
Hence, the central elements
$\{p_{i,j}^{(rp)};~r > 0\}$ give another algebraically
independent set of generators for
$Z(Y_{m|n})\cap Y_{i,j}^+$
lifting the central elements $\{(e_{i,j}x^{r-1})^p;~r >
0\}$ of $\gr Y_{m|n}$.
Analogous statements with $Y_{i,j}^+, e, p$ and $e_{i,j}x^{r-1}$
replaced by $Y_{j,i}^-, f ,q$ and $e_{j,i}x^{r-1}$ also hold.
\end{Theorem}
\begin{proof}
The proof, which uses Theorem \ref{center of enveloping algebra of shifted current algebra},
Lemma \ref{root power series central} and Lemma \ref{root subalgebra p center},
is similar to the proof of \cite[Theorem 5.4]{BT18}.
\end{proof}

\subsection{\boldmath Diagonal $p$-central elements}\label{subsection: diagonal p-central elements}
This subsection we introduce the $p$-central elements that belong to the diagonal subalgebras
$$Y_i^0:= \kk[d_i^{(r)};~r>0]$$
of $Y_{m|n}$ for $1\leq i\leq m+n$.
Note that  $Y_i^0=\kk[d_i'^{(r)};~r>0]$ by (\ref{di di' relation}).

We define
\begin{equation}\label{definition b}
b_i(u):=\sum_{r \geq 0}b_i^{(r)}u^{-r}
:=\left\{
\begin{array}{ll}
d_{i\downarrow p}(u)=d_i(u)d_i(u-1)\cdots d_i(u-p+1)&\text{if $|i|=0$,}\\
d_{i\downarrow p}(u)^{-1}=d_i(u)^{-1}d_i(u-1)^{-1}\cdots d_i(u-p+1)^{-1}&\text{if $|i|=1$.}
\end{array}\right.
\end{equation}

\begin{Lemma}\label{diag center lemma 1}
For all $i=1,\dots,m+n$ and $r>0$, the elements $b_i^{(r)}$ belongs to $Z(Y_{m|n})$.
\end{Lemma}
\begin{proof}
In view of Theorem \ref{Ymn generated by fde} and (\ref{di dj commu}),
it suffices to check that
$$[b_i(u),e_j(v)]=0=[b_i(u),f_j(v)]$$
for all $1\leq j\leq m+n-1$.
By applying the anti-automorphism $\tau$ (Section \ref{Section Automorphisms}(2)),
it suffices to check just the first equality.
By (\ref{identities in Ymn-3}), the first equality obviously holds when $j\notin\{i-1,i\}$.
Consider first the case that $j=i-1$.
When $|i|=0$, then (\ref{d arrow e 3}) implies that
$$[b_i(u),e_{i-1}(v)]=[d_{i\downarrow p}(u),e_{i-1}(v)]=[d_{i\uparrow p}(u-p+1),e_{i-1}(v)]=0.$$
Now let $|i|=1$, then (\ref{d arrow e 4}) yields $[d_{i\downarrow p}(u),e_{i-1}(v)]=0$,
so that $$[b_i(u),e_{i-1}(v)]=[d_{i\downarrow p}(u)^{-1},e_{i-1}(v)]=0.$$
One argues similarly for the case $j=i$ using (\ref{d arrow e 1}), (\ref{d arrow e 2}) instead of (\ref{d arrow e 3}), (\ref{d arrow e 4}).
\end{proof}

\begin{Theorem}\label{theorem dig center}
For $1\leq i\leq m+n$,
the algebra $Z(Y_{m|n})\cap Y_i^0$
is an infinite rank polynomial algebra freely generated by
the central elements $\{b_i^{(rp)};~r > 0\}$.
We have that $b_i^{(rp)} \in{\rm F}_{rp-p}Y_{m|n}$
and
\begin{equation}\label{grading of brp}
\gr_{rp-p}b_i^{(rp)}=(e_{i,i}x^{r-1})^p-e_{i,i}x^{rp-p}.
\end{equation}
For $0<r <p$, we have that
$b_i^{(r)}=0$. For $r>p$ with $p \nmid
r$, we have that $b_i^{(r)}\in{\rm F}_{r-p-1}Y_{m|n}$ and it
is a polynomial in the elements $\{b_i^{(sp)};~0<s\leq \lfloor r/p \rfloor\}$.
\end{Theorem}
\begin{proof}
Recall that $\gr_rd_i^{(r+1)}=-\gr_rd_i'^{(r+1)}$,
so that the identification (\ref{identification}) implies that $\gr_rd_i'^{(r+1)}=e_{i,i}x^r$ when $|i|=1$.
Then the proof, which uses Theorem \ref{center of enveloping algebra of shifted current algebra},
Lemma \ref{diag center lemma 1},
is similar to the proof of \cite[Theorem 5.8]{BT18}.
\end{proof}

\subsection{\boldmath The center $Z(\Ymn)$}
We define the {\it $p$-center} $Z_p(Y_{m|n})$
of $Y_{m|n}$ to be the subalgebra generated by
\begin{equation}\label{p-center}
\big\{b_i^{(rp)};~1\leq i \leq m+n, r > 0\big\} \cup
\Big\{\big(e_{i,j}^{(r)} \big)^p,
\big(f_{j,i}^{(r)} \big)^p;~1\leq i<j\leq m+n, r>0, |i|+|j|=0\Big\}.
\end{equation}
According to Proposition \ref{HC center of Ymn}, Lemma \ref{root subalgebra p center} and Lemma \ref{diag center lemma 1},
we know that both $Z_\HC(Y_{m|n})$ and $Z_p(Y_{m|n})$ are subalgebras of $Z(Y_{m|n})$.
Note also by (\ref{identification}) and (\ref{grading of brp})
that $\gr Z_p(Y_{m|n})$ may be identified with the $p$-center $Z_p(\fg)$ of $U(\fg)$ from (\ref{p-center is free polynomial}).

We need one more family of elements.
Recalling (\ref{definition c}) and
(\ref{definition b}), we let
\begin{eqnarray}\label{definition bc}
bc(u)&:=&\sum_{r\geq 0}bc^{(r)}u^{-r}\nonumber\\
&:=&b_1(u)b_2(u-1)\cdots b_m(u-m+1)b_{m+1}(u-m+1)\cdots b_{m+n}(u-m+n)\nonumber\\
&=&c(u)c(u-1)\cdots c(u-p+1).
\end{eqnarray}
By definition each $bc^{(r)}$ can be expressed as a polynomial in the elements
$\{c^{(s)};~s > 0\}$, so that it belongs to
$Z_{\HC}(Y_{m|n})$.
It is also a polynomial in the elements
$\{b_i^{(s)};~1\leq i\leq m+n, s > 0\}$, so that it belongs to $Z_p(Y_{m|n})$ by Theorem \ref{theorem dig center}.
Consequently, $bc^{(r)}\in Z_{\HC}(Y_{m|n})\cap Z_p(Y_{m|n})$.

\begin{Lemma}
For $r > 0$,
we have that $bc^{(rp)} \in{\rm F}_{rp-p}Y_{m|n}$
and
\begin{equation}\label{grading of bc rp}
\gr_{rp-p}bc^{(rp)}=z_{r-1}^p-z_{rp-p}.
\end{equation}
\end{Lemma}
\begin{proof}
By definition of $bc(u)$ (\ref{definition bc}), we have
$$bc(u)=\prod_{i=1}^{p}\left(\sum_{r\geq 0}c^{(r)}(u - i+1)^{-r}\right).$$
Since $c^{(0)}=1$ and the set $\{c^{(r)};~r>0\}$ are commuting elements by Proposition \ref{HC center of Ymn},
it follows from \cite[Lemma 2.9]{BT18} that
$bc^{(rp)}\in{\rm F}_{rp-p}Y_{m|n}$
and
$$
bc^{(rp)}\equiv(c^{(r)})^p-c^{(rp-p+1)}~({\rm mod}~{\rm F}_{rp-p-1}Y_{m|n}).
$$
Our assertion now follows from the fact that ${\rm gr}_{r-1}c^{(r)}=z_{r-1}$.
\end{proof}

Now we can state the main result of this section.
The foregoing observations in conjunction with (\ref{p-center is free polynomial}), (\ref{identification}) and Theorem \ref{center of enveloping algebra of shifted current algebra} yield the following Theorem.
The proof is similar to the proof of \cite[Theorem 5.11]{BT18}.

\begin{Theorem}\label{main theorem: center of Ymn}
The centre $Z(Y_{m|n})$ is generated by $Z_{\HC}(\Ymn)$ and
$Z_p(\Ymn)$. Moreover:
\begin{enumerate}
\item $Z_{\HC}(\Ymn)$ is the free polynomial algebra
generated by $\{c^{(r)};~r > 0\}$;
\item $Z_p(\Ymn)$ is the free polynomial algebra
generated by
\begin{equation}\label{generator of p center}
\big\{b_i^{(rp)};~1\leq i\leq m+n, r > 0\big\}\cup
\Big\{\big(e_{i,j}^{(r)}\big)^p,
\big(f_{j,i}^{(r)}\big)^p;~1\leq i< j\leq m+n, r>0,|i|+|j|=0\Big\};
\end{equation}
\item $Z(\Ymn)$ is the free polynomial algebra generated by
\begin{equation}
\label{generator of center of Ymn}
\big\{b_i^{(rp)}, c^{(r)};~2\leq i\leq m+n,r>0\big\}\cup\Big\{\big(e_{i,j}^{(r)} \big)^p,
\big(f_{i,j}^{(s)} \big)^p;~1\leq i< j\leq m+n, r>0,|i|+|j|=0\Big\};
\end{equation}
\item $Z_{\HC}(Y_{m|n})\cap Z_p(Y_{m|n})$ is the free polynomial algebra generated
by $\{bc^{(rp)};~r>0\}$.
\end{enumerate}
\end{Theorem}

\begin{Corollary}\label{main coro 1}
The super Yangian $\Ymn$ is free as a module over its center, with
basis given by the ordered supermonomials in
\begin{equation}\label{supermonomial}
\big\{d_i^{(r)};~2 \leq i \leq m+n, r >
0\big\}\cup\big\{e_{i,j}^{(r)},
f_{j,i}^{(r)};~1\leq i < j \leq m+n, r>0\big\}
\end{equation}
in which no exponent is $p$ or more for $d_i^{(r)}$ and $e_{i,j}^{(r)},
f_{j,i}^{(r)}$ with $|i|+|j|=0$.
\end{Corollary}
\begin{proof}
Let $Y_{m|n}^0$ denote the subalgebra of $Y_{m|n}$ generated by the elements $\{d_i^{(r)}\}$.
We consider the sets
\begin{align}\label{set 1}
\{b_i^{(rp)}, c^{(r)};~2\leq i\leq m+n,r>0\}
\end{align}
and
\begin{align}\label{set 2}
\{d_i^{(r)};~2 \leq i \leq m+n, r >0\}.
\end{align}
It suffices to show that the set consisting of ordered monomials in (\ref{set 1})
multiplied by ordered monomials in (\ref{set 2}) with exponents $<p$ gives a basis for $Y_{m|n}^0$.
To see this, we pass to the associated graded algebra using
(\ref{identification}), (\ref{grading of brp}) and Proposition \ref{HC center of Ymn}
to reduce to showing that the monomials
$$
\prod_{r \geq 0}
z_r^{a_{1,r}}
\prod_{\substack{2 \leq i\leq m+n\\r \geq 0}}
\big((e_{i,i} t^r)^p - e_{i,i}t^{rp}\big)^{a_{i,r}}\!\!
\prod_{\substack{2\leq i \leq m+n\\r \geq 0}}
(e_{i,i} t^r)^{b_{i,r}}
$$
for $a_{i,r}\geq 0$ and $0\leq b_{i,r}<p$ form a basis for $\gr Y_{m|n}^0$.
This is quite straightforward: these monomials are related to
a PBW basis of $\gr Y_{m|n}^0$ by a uni-triangular transition matrix.
\end{proof}

Similarly, we have:
\begin{Corollary}\label{main coro 2}
The super Yangian $\Ymn$ is free as a module over $Z_p(\Ymn)$ with basis
given by the ordered supermonomials in
$$
\big\{d_i^{(r)};~1\leq i\leq m+n, r >
0\big\}\cup\big\{e_{i,j}^{(r)},
f_{j,i}^{(r)};~1\leq i < j \leq m+n, r>0\big\}
$$
in which no exponent is $p$ or more for $d_i^{(r)}$ and $e_{i,j}^{(r)},
f_{j,i}^{(r)}$ with $|i|+|j|=0$.
\end{Corollary}

\section{\boldmath Modular super Yangian of $\mathfrak{sl}_{m|n}$}\label{section:special super Yangian}
In this section, we define the modular version of the super Yangian of $\mathfrak{sl}_{m|n}$,
which is a subalgebra $\SYmn$ of $\Ymn$.
We give a presentation for $\SYmn$ valid in any characteristic by using the diagonal elements defined in (\ref{definition h}).
We will show that this presentation is equivalent to the usual Drinfeld-type presentation (see \cite[Proposition]{Gow07}, \cite[Section 2.5]{Tsy20})
whenever $\Char\kk\neq 2$.
\subsection{The special super Yangian}
We define the special super Yangian associated to the special linear Lie superalgebra $\mathfrak{sl}_{m|n}$ as the following subalgebra of $\Ymn$:
\begin{equation}\label{definition of SYmn}
SY_{m|n}:= \{x\in Y_{m|n};~\mu_f(x)=x~\text{for all }f(u) \in 1+u^{-1} \kk[[u^{-1}]]\},
\end{equation}
where we take $\mu_f$ as defined as in Section \ref{Section Automorphisms}(1).

Let $\fg':=\mathfrak{sl}_{m|n}[x]$ be the current superalgebra associated to $\mathfrak{sl}_{m|n}$.
The following theorem is a generalization and modular analogue of \cite[Proposition 3, Lemma 7]{Gow07}.
\begin{Theorem}\label{PBW SYmn}
The algebra $\SYmn$ has a basis consisting of ordered supermonomials in
\begin{equation}\label{Symn theorem 1-1}
\big\{h_i^{(r)};~1\leq i <m+n, r > 0\big\}\cup
\big\{e_{i,j}^{(r)}, f_{j,i}^{(r)};~1 \leq i< j \leq m+n, r
> 0\big\}
\end{equation}
taken in any fixed order.
Hence, $\gr\SYmn = U(\fg')$,
and multiplication defines a vector space isomorphism
\begin{equation}\label{Symn theorem 1-2}
\SYmn \otimes Y_1 \stackrel{\sim}{\rightarrow} Y_{m|n}
\end{equation}
where $Y_1$ is identified with the subalgebra of $Y_{m|n}$ generated by
the elements
$d_1^{(r)}$ in the obvious way.
If we assume in addition that $p \nmid (m-n)$ then multiplication defines an algebra isomorphism
\begin{equation}\label{Symn theorem 1-3}
\SYmn \otimes Z_{\HC}(Y_{m|n}) \stackrel{\sim}{\rightarrow}\Ymn.
\end{equation}
\end{Theorem}
\begin{proof}
The proof is the same as in the non-super case (see \cite[Theorem 6.1]{BT18}).
We just give a brief indication.

First note that all $h_i^{(r)}$ belong to $\SYmn$ by the definitions of $h_i(u)$ (\ref{definition h}) and the automorphism $\mu_f$ (Section \ref{Section Automorphisms}(1)),
and of course all $e_{i,j}^{(r)}$ and $ f_{j,i}^{(r)}$ belong to $\SYmn$ too.
Let $\overline{\SYmn}$ be the subspace of $\SYmn$ spanned by the ordered supermonomials in the elements (\ref{Symn theorem 1-1}).
Passing to the associated graded space induced by the filtration of $\Ymn$ and using (\ref{identification}) and (\ref{grading of hi r+1}),
it follows that $\gr\overline{\SYmn}=U(\fg')$ and thus the ordered supermonomials that span $\overline{\SYmn}$ are in particular linearly independent.
Then, multiplying them by ordered monomials in $\{d_1^{(r)};~r>0\}$ gives the following isomorphism:
$$\overline{\SYmn}\otimes k[d_1^{(r)};~r>0]\cong\Ymn.$$
Furthermore, we can use the above isomorphism to show that the inclusion $\SYmn\subseteq\overline{\SYmn}$.
This proves the isomorphism (\ref{Symn theorem 1-2}).
Finally, the assumption $p \nmid (m-n)$ ensures that the elements
$$\{e_{i,i}x^r-(-1)^{|i|+|i+1|}e_{i+1,i+1}x^{r};~1\leq i<m+n\}\cup\{z_r;~r\geq 0\}\cup\{e_{i,j} x^r;~1 \leq i\neq j \leq m+n, r \geq 0\}$$
form a basis of $\fg$. By considering the associated graded algebra, the isomorphism (\ref{Symn theorem 1-3}) can be proven similarly.
\end{proof}

In view of (\ref{induction eij,fij}),
Theorem \ref{PBW SYmn} implies that $\SYmn$ can be generated by the elements $\{h_i^{{r}},e_i^{(r)},f_{i}^{(r)}\}$.
Then we have the following presentation for the subalgebra $\SYmn$.

\begin{Theorem}\label{Theorem: presentaiton of SYmn}
The algebra $\SYmn$ is generated by the elements
\begin{equation}\label{beany}
\big\{h_i^{(r)}, e_i^{(r)}, f_i^{(r)};~1\leq i <m+n, r > 0\big\}
\end{equation}
subject only to the relations (\ref{Drinfeld generators relation 4})-(\ref{Drinfeld generators relation 13}) plus the following:
\begin{align}
\big[h_i^{(r)}, h_j^{(s)}\big] &= 0,\label{SYmn presentation relation 1}\\
\big[e_i^{(r)}, f_j^{(s)}\big] &=(-1)^{|i|+|i+1|}\delta_{i,j}h_i^{(r+s-1)},\label{SYmn presentation relation 2}\\
\big[h_i^{(r)}, e_j^{(s)}\big] &= 0\hspace{18mm}\text{if $|i-j|>1$},\label{SYmn presentation relation 3}\\
\big[h_i^{(r)}, f_j^{(s)}\big] &= 0\hspace{18mm}\text{if $|i-j|>1$},\label{SYmn presentation relation 4}\\
\big[h_{i-1}^{(r+1)}, e_i^{(s)}\big]-\big[h_{i-1}^{(r)}, e_i^{(s+1)}\big]&=(-1)^{|i|}h_{i-1}^{(r)}e_i^{(s)},\label{SYmn presentation relation 5}\\
\big[h_{i-1}^{(r)}, f_i^{(s+1)}\big]-\big[h_{i-1}^{(r+1)}, f_i^{(s)}\big]&=(-1)^{|i|}f_i^{(s)}h_{i-1}^{(r)},\label{SYmn presentation relation 6}\\
\big[h_{i}^{(r+1)}, e_i^{(s)}\big]-\big[h_{i}^{(r)}, e_i^{(s+1)}\big]&=\left\{
\begin{array}{ll}
-(-1)^{i+1}\left(h_{i}^{(r)}e_i^{(s)}+e_i^{(s)}h_{i}^{(r)}\right),&\text{if $i\neq m$},\label{SYmn presentation relation 7}\\
0&\text{if $i=m$},\\
\end{array}\right.\\
\big[h_{i}^{(r)}, f_i^{(s+1)}\big]-\big[h_{i}^{(r+1)},f_i^{(s)}\big]&=\left\{
\begin{array}{ll}
-(-1)^{i+1}\left(f_{i}^{(s)}h_i^{(s)}+h_{i}^{(r)}f_i^{(s)}\right),&\text{if $i\neq m$},\label{SYmn presentation relation 8}\\
0&\text{if $i=m$},\\
\end{array}\right.\\
\big[h_{i+1}^{(r+1)}, e_i^{(s)}\big]-\big[h_{i+1}^{(r)},e_i^{(s+1)}\big]&=(-1)^{|i+1|}e_i^{(s)}h_{i+1}^{(r)},\label{SYmn presentation relation 9}\\
\big[h_{i+1}^{(r)}, f_i^{(s+1)}\big]-\big[h_{i+1}^{(r+1)}, f_i^{(s)}\big]&=(-1)^{|i+1|}h_{i+1}^{(r)}f_i^{(s)},\label{SYmn presentation relation 10}
\end{align}
for all admissible $i,j,r,s$ including $r=0$ in
(\ref{SYmn presentation relation 5})--(\ref{SYmn presentation relation 10}); remember also $h_i^{(0)}=-(-1)^{|i|}$.
\end{Theorem}
\begin{proof}
Clearly, (\ref{Drinfeld generators relation 4})-(\ref{Drinfeld generators relation 13}) hold,
and the relations (\ref{SYmn presentation relation 1}),
(\ref{SYmn presentation relation 3})-(\ref{SYmn presentation relation 4})
follow from (\ref{di dj commu})-(\ref{Drinfeld generators relation 2}).
Referring to the definition (\ref{definition h}), we have
$$h_i^{(r+s-1)}=-(-1)^{|i|}\sum_{t=0}^{r+s-1}d_i'^{(t)}d_{i+1}^{(r+s-1-t)}.$$
Now the relation (\ref{SYmn presentation relation 2}) follows directly from (\ref{Drinfeld generators relation 3}).
For the remaining relations, (\ref{SYmn presentation relation 5}), (\ref{SYmn presentation relation 7}) and (\ref{SYmn presentation relation 9})
follow by taking coefficients of $u^{-r}v^{-s}$
in (\ref{h e bracket relation 3}), (\ref{coro 1}) and (\ref{h e bracket relation 5}), respectively.
Then (\ref{SYmn presentation relation 7}), (\ref{SYmn presentation relation 8}) and (\ref{SYmn presentation relation 10}) follow by applying the anti-automorphism $\tau$.
For the rest of the proof, one can show by the same argument as \cite[Theorem 6.3]{BT18}.
\end{proof}

Suppose that $\Char\kk=p\neq 2$.
We have the Drinfeld presentation for the super Yangian $\SYmn$ (cf. \cite[Definition 2]{Stu94}, \cite[Proposition 5]{Gow07} and \cite[Section 2.5]{Tsy20}). We shall use the ``opposite'' presentation (See \cite[Remark 5.12]{BK05} for the Yangian $Y(\mathfrak{sl}_n)$).
In more detail,
we define $\kappa_{i,s},\xi_{i,s}^{\pm}$ for $i=1,\dots,m+n-1$ and $s\geq 0$ from the equations
\begin{align*}
\kappa_i(u)=\sum_{s \geq 0} \kappa_{i,s}u^{-s-1}&:=(-1)^{|i|}+\eta_{(-1)^{|i|}(i-m)/2}(h_i(u)),\\
\xi_i^+(u)=\sum_{s \geq 0} \xi_{i,s}^+ u^{-s-1} &:=\eta_{(-1)^{|i|}(i-m)/2}(e_i(u)),\\
\xi_i^-(u)=\sum_{s \geq 0} \xi_{i,s}^- u^{-s-1} &:=\eta_{(-1)^{|i|}(i-m)/2}(f_i(u)).
\end{align*}
To ease notation we let $(a_{i,j})_{i,j=1}^{m+n-1}$ be the Cartan matrix associated to the Lie superalgebra $\mathfrak{sl}_{m|n}$,
that is, $a_{i,j}=(-1)^{|i|}(\delta_{i,j}-\delta_{i,j+1})-(-1)^{|i+1|}(\delta_{i+1,j}-\delta_{i+1,j+1})$.

\begin{Proposition}
The Yangian $\SYmn$ is generated by the elements $\{\kappa_{i,s},\xi_{i,s}^{\pm};~1\leq i<m+n, s\geq 0\}$ subject only to the following relations:
\begin{align}
\big[\kappa_{i,r}, \kappa_{j,s}\big] &= 0, \label{D-SYmn-1}\\
\big[\xi^+_{i,r}, \xi^-_{j,s}\big] &= (-1)^{|i|+|i+1|}\delta_{i,j} \kappa_{i,r+s},\label{D-SYmn-2}\\
\big[\kappa_{i,0}, \xi^{\pm}_{j,s}\big] &= \pm(-1)^{|i|} a_{i,j} \xi^{\pm}_{j,s},\label{D-SYmn-3}\\
\big[\kappa_{i,r},\xi^\pm_{j,s+1}\big] - \big[\kappa_{i,s+1}, \xi^{\pm}_{j,l}\big]
&= \pm \frac{a_{i,j}}{2} (\kappa_{i,r} \xi^{\pm}_{j,s} + \xi^{\pm}_{j,s} \kappa_{i,r}),~\text{for $i,j$ not both $m$},\label{D-SYmn-4}\\
\big[\kappa_{m,r+1},\xi^{\pm}_{m,s}\big]&=0,\label{D-SYmn-5}\\
\big[\xi^\pm_{i,r}, \xi^\pm_{j,s+1}\big] - \big[\xi^\pm_{i,r+1}, \xi^\pm_{j,s}\big]
&= \pm \frac{a_{i,j}}{2} (\xi^\pm_{i,r} \xi^{\pm}_{j,s} +
\xi^\pm_{j,s} \xi^\pm_{i,r}),~\text{for $i,j$ not both $m$}\label{D-SYmn-6}\\
\big[\xi^{\pm}_{m,r},\xi^{\pm}_{m,s}\big]&=0,\label{D-SYmn-7}\\
\big[\xi_{i,r}^{\pm}, \xi_{j,s}^\pm\big] &= 0\text{ if $|i-j|>1$,}\label{D-SYmn-8}\\
\Big[\xi_{i,r}^{\pm}, \big[\xi_{i,s}^{\pm}, \xi_{j,t}^\pm\big]\Big]
+
\Big[\xi_{i,s}^{\pm}, \big[\xi_{i,r}^{\pm}, \xi_{j,t}^\pm\big]\Big]
&= 0\text{ if $|i-j|=1$,}\label{D-SYmn-9}\\
\Big[\big[\xi_{m-1,r}^{\pm},\xi_{m,0}^{\pm}\big],\big[\xi_{m,0}^{\pm}, \xi_{m+1,r}^{\pm}\big]\Big]&=0.\label{D-SYmn-10}
\end{align}
\end{Proposition}
\begin{proof}
The proof is just a rephrasing of Theorem \ref{Theorem: presentaiton of SYmn} for these generators by taking some special coefficients in the power series.
For example, the relation (\ref{D-SYmn-2}) follows from (\ref{identities in Ymn-5}) in conjunction with the identities
$$\frac{\kappa_i(u)-\kappa_i(v)}{v-u}=\sum\limits_{r,s\geq 0}\kappa_{i,r+s-1}u^{-r-1}v^{-s-1};$$
To check (\ref{D-SYmn-3}) in the case $i=j+1$ and the sign is $+$:
we set $u':=u-\frac{(-1)^{|i+1|}}{2}(m-i)$ and $v':=v-\frac{(-1)^{|i|}}{2}(m-i)$ in (\ref{coro3}).
Then taking the $v'^{-s-1}$-coefficient yields
$$[\kappa_{i+1,0},\xi^+_{i,s}]=-1=(-1)^{|i+1|}a_{i+1,i}\xi^+_{i,s};$$
The relations (\ref{D-SYmn-4})-(\ref{D-SYmn-5}) follow from (\ref{coro 1})-(\ref{coro3}) and (\ref{D-SYmn-6})
follows from (\ref{identities in Ymn-6})-(\ref{identities in Ymn-9}).
As $p\neq 2$ and both $e_m$ and $f_m$ are odd, the relations (\ref{identities in Ymn-6})-(\ref{identities in Ymn-7})
yield $$(u-v)[e_m(u), e_m(v)]=0=(u-v)[f_m(u), f_m(v)],$$
so the relation (\ref{D-SYmn-7}) follows;
the relation (\ref{D-SYmn-8}) can be deduced from (\ref{identities in Ymn-2});
Finally, (\ref{D-SYmn-9})-(\ref{D-SYmn-10}) follow from (\ref{identities in Ymn-12})-(\ref{identities in Ymn-15}).
\end{proof}

\begin{Remark}
(1) Note that $p\neq 2$.
We may set $r=s$ in (\ref{D-SYmn-9}) to see that
$$\big[\xi_{i,r}^{\pm}, [\xi_{i,r}^{\pm}, \xi_{j,t}^\pm]\big]=0\text{ if $|i-j|=1$}.$$
For the quartic Serre relations (\ref{D-SYmn-10}),
we just consider the case for $i=m$ (See Remark \ref{rk2}).

(2) The case $i=j$ in (\ref{D-SYmn-6}) and (\ref{D-SYmn-7}) are equivalent to the following relations:
\begin{align*}
[\xi_{i,r}^+,\xi_{i,s}^+]=(-1)^{|i+1|}(\sum_{t=0}^{s-1}\xi_{i,t+r}^+,\xi_{i,s-t-1}^+-\sum_{t=0}^{r-1}\xi_{i,t+s}^+,\xi_{i,r-t-1}^+),
\end{align*}
\begin{align*}
[\xi_{i,r}^-,\xi_{i,s}^-]=(-1)^{|i+1|}(\sum_{t=0}^{r-1}\xi_{i,t+s}^-,\xi_{i,r-t-1}^--\sum_{t=0}^{s-1}\xi_{i,t+r}^-,\xi_{i,s-t-1}^-).
\end{align*}
\end{Remark}

\subsection{\boldmath The $p$-centre of $\SYmn$}
Let
\begin{align}\label{definition of a}
a_i(u)&=\sum_{r \geq 0}a_i^{(r)} u^{-r}:=h_i(u)h_i(u-1)\cdots h_i(u-p+1)\nonumber\\
&=\left\{
\begin{array}{ll}
-b_{i+1}(u)b_i(u)^{-1}&\text{if $i<m$,}\\
-b_{i+1}(u)^{-1}b_i(u)^{-1}&\text{if $i=m$,}\\
b_{i+1}(u)^{-1}b_i(u)^{-1}&\text{if $i>m$,}
\end{array}\right.
\end{align}
where the last equality follows from the definition ((\ref{definition h}) and (\ref{definition b})).
According to Lemma \ref{diag center lemma 1}, each $a_i^{(r)}$ belongs to the center of $\SYmn$.
We define the {\it $p$-center} of $\SYmn$ to be the subalgebra
$Z_p(\SYmn)$ of
$Z(\SYmn)$ generated by
\begin{equation}\label{generators of  p-center of SYmn}
\big\{a_i^{(rp)};~1 \leq i<m+n, r > 0\big\}\cup
\Big\{\big(e_{i,j}^{(r)} \big)^p,
\big(f_{j,i}^{(r)} \big)^p;~1\leq i<j\leq m+n, r>0, |i|+|j|=0\Big\}.
\end{equation}
We also let
$$Z_p(\fg'):=\kk\langle x^p-x^{[p]};~x\in(\fg')_{\bar{0}}\rangle$$
be the $p$-center of $U(\fg')$.

\begin{Theorem}\label{p-center symn theorem 1}
The generators (\ref{generators of  p-center of SYmn}) of $Z_p(\SYmn)$ are algebraically
independent,
and we have that $\gr Z_p(SY_n) = Z_p(\fg')$.
Moreover, $\SYmn$
is free as a module over $Z_p(SY_n)$ with
basis given by the ordered supermonomials in
\begin{equation*}
\big\{h_i^{(r)};~1 \leq i < m+n, r >
0\big\}\cup\big\{e_{i,j}^{(r)},
f_{i,j}^{(r)};~1\leq i < j \leq m+n, r > 0\big\}
\end{equation*}
in which no exponent is $p$ or more for $h_i^{(r)}$ and $e_{i,j}^{(r)},
f_{j,i}^{(r)}$ with $|i|+|j|=0$.
\end{Theorem}
\begin{proof}
We put
$$\widetilde a_i(u):=-(-1)^{|i|}a_i(u)=\sum_{r \geq 0}\widetilde a_i^{(r)}u^{-r}$$ and $$\widetilde h_i(u):=-(-1)^{|i|}h_i(u)=\sum_{r \geq 0}\widetilde h_i^{(r)}u^{-r}.$$
By (\ref{definition of a}), we have
$$\widetilde a_i(u)=\prod_{j=1}^{p}\left(\sum_{r\geq 0}\widetilde h_i^{(r)}(u -j+1)^{-r}\right).$$
Note that $\widetilde h_i^{(0)}=1$ and the set $\{\widetilde h_i^{(r)};~r>0\}$ are commuting elements by (\ref{di dj commu}).
Then \cite[Lemma 2.9]{BT18} implies that
$\widetilde a_i^{(rp)}\in{\rm F}_{rp-p}Y_{m|n}$
and
$$
\widetilde a_i^{(rp)}\equiv(\widetilde h_i^{(r)})^p-\widetilde h_i^{(rp-p+1)}~({\rm mod}~{\rm F}_{rp-p-1}Y_{m|n}).
$$
Now (\ref{grading of hi r+1}) readily yields
\begin{equation*}
\gr_{rp-p}a_i^{(rp)}=\big(e_{i,i}x^{r-1}-(-1)^{|i|+|i+1|}e_{i+1,i+1}x^{r-1}\big)^p-\big(e_{i,i}x^{rp-p}-(-1)^{|i|+|i+1|}e_{i+1,i+1}x^{rp-p}\big).
\end{equation*}
In conjunction with (\ref{center 1-111}) this implies that the generators (\ref{generators of  p-center of SYmn}) of $Z_p(\SYmn)$
are lifts of generators for $Z_p(\fg')$ coming from a basis for $\fg'$.
This establishes the algebraic independence and that $\gr Z_p(\SYmn)=Z_p(\fg')$.
The final assertion follows by similar argument to the proof of Corollary \ref{main coro 1}
using the PBW basis for $\SYmn$ from Theorem \ref{PBW SYmn}.
\end{proof}

\begin{Proposition}\label{Proposition ZpSYmm=ZYmm}
If $p\nmid (m-n)$, then $Z_p(\SYmn)=Z(\SYmn)$.
\end{Proposition}
\begin{proof}
Obviously,
$Z_p(\SYmn)\subseteq Z(\SYmn)$.
When combined with Theorem \ref{p-center symn theorem 1}, this implies
$$Z_p(\fg')=\gr Z_p(\SYmn)\subseteq\gr Z(\SYmn)\subseteq Z(\fg').$$
It suffices to verify that $Z_p(\fg')=Z(\fg')$.
Our assumption on $p$ ensures that $\fg=\fg'\oplus\fz(\fg)$.
Hence, $Z(\fg)\cong Z(\fg')\otimes\kk[z_r;~r\geq 0]$.
Moreover,
Theorem \ref{center of enveloping algebra of shifted current algebra} in combination with the assumption $p\nmid (m-n)$ shows that
the elements $\{x^p-x^{[p]};~x\in\fg'\}\cup\{z_r;~r\geq 0\}$ generate $Z(\fg)$.
This implies the assertion.
\end{proof}

\subsection{\boldmath Another description of the $p$-center of $\Ymn$}
Given $|i|+|j|=0$, we consider the even elements $\{t_{i,j}^{(r)};~r\geq 0\}$.
Using the defining relation (\ref{RTT relations}) and induction on $r+s$,
it is easy to see that
\begin{equation*}
t_{i,j}^{(r)}t_{i,j}^{(s)}=t_{i,j}^{(s)}t_{i,j}^{(r)}
\end{equation*}
for all $r,s \geq 0$.
For $|i|+|j|=0$,
we define
\begin{equation}\label{definition s-ij}
s_{i,j}(u) = \sum_{r \geq 0}s_{i,j}^{(r)} u^{-r}
:=\left\{
\begin{array}{ll}
t_{i,j}(u)t_{i,j}(u-1)\cdots t_{i,j}(u-p+1)&\text{if $|i|=0$,}\\
t_{i,j}'(u)t_{i,j}'(u-1)\cdots t_{i,j}'(u-p+1)&\text{if $|i|=1$.}
\end{array}\right.
\end{equation}
The foregoing observations in conjunction with Proposition \ref{proposition auto}(4)
imply that the order of the product on the right hand side here is irrelevant.

\begin{Lemma}\label{Lemma: sij coeff in p-center}
All of the elements $s_{i,j}^{(r)}$ belong to the $p$-center
$Z_p(\Ymn)$.
\end{Lemma}
\begin{proof}
First we show that each $s_{i,j}^{(r)}$ belongs to $Z(\Ymn)$.
The general assumption $|i|+|j|=0$ then implies that we need to consider the two cases:
(a) $1\leq i,j\leq m$; (b) $m+1\leq i,j\leq m+n$.
For (a), one uses the permutation automorphism from Section \ref{Section Automorphisms}(3) to reduce the problem to showing that
$$(\dag)\ \ \ \ \text{all coefficients of $s_{1,1}(u)$ and of $s_{1,2}(u)$ are central}.$$
Using the swap map (Proposition \ref{proposition auto}) one can show by direct computation that
$$\zeta_{m|n}(t'_{i,j}(u))=t_{m+n+1-i,m+m+1-j}(u).$$
If $m+1\leq i,j\leq m+n$, then $\zeta_{m|n}(s_{i,j}(u))=s_{m+n+1-i,m+m+1-j}(u)\in Y_{n|m}$.
We can also use the permutation automorphism to reduce case (b) to the problem of showing $(\dag)$.
Since the coefficients of $b_i(u)$ and $p_{i,j}$ are contained in the $p$-center of $\Ymn$,
the $(\dag)$ follows because the following claim.
\begin{align}
s_{1,1}(u) &= b_1(u),\label{claim 1}\\
s_{1,2}(u) &= b_1(u) p_{1,2}(u).\label{claim 2}
\end{align}
The definition of $b_i(u)$ (\ref{definition b}) in combination with (\ref{quasideterminants D}) gives $t_{1,1}(u)=d_1(u)$,
so that the first identity (\ref{claim 1}) follows.
For (\ref{claim 2}), we set $i=1$ and $v=u-k$ in (\ref{d arrow e 1}) to deduce that
\begin{equation*}
e_1(u-k)d_1(u-k+1)\cdots d_1(u-1)d_1(u)=d_1(u-k+1)\cdots d_1(u-1)d_1(u)e_1(u)
\end{equation*}
for each $k=1,\dots,p-1$, while the Gauss decomposition (\ref{gauss decomp}) yields $t_{1,2}(u)=d_1(u)e_1(u)$.
Consequently,
\begin{align*}
s_{1,2}(u) &=t_{1,2}(u-p+1)\cdots t_{1,2}(u-1)t_{1,2}(u)\\
&=
d_1(u-p+1)e_1(u-p+1) \cdots d_1(u-1) e_1(u-1)d_1(u) e_1(u)
\\&=
d_1(u-p+1) \cdots d_1(u-1) d_1(u) e_1(u)^p=b_1(u) p_{1,2}(u).
\end{align*}
This establishes (\ref{claim 2}).

By Theorem \ref{main theorem: center of Ymn}(2), we have that
$Z_p(\Ymn)=Y_{m|n}\cap Z_p(Y_{m|n+1})$,
where we are using the natural embedding $Y_{m|n}\hookrightarrow Y_{m|n+1};~t_{i,j}^{(r)} \mapsto t_{i,j}^{(r)}$.
In order to prove that $s_{i,j}^{(r)} \in Z_p(\Ymn)$,
we may thus assume that that $p \nmid (m-n)$.
Equivalently, we show that $s_{i,j}(u) \in Z_p(\Ymn)[[u^{-1}]]$.
This is immediate by (\ref{claim 1}) in case $i=j=1$.
In general, we will show that $s_{i,j}(u)s_{1,1}(u)^{-1}\in Z_p(\Ymn)[[u^{-1}]]$.
Using the definition (\ref{definition of SYmn}),
we get that $s_{i,j}(u)s_{1,1}(u)^{-1} \in\SYmn[[u^{-1}]]$.
Since we have shown its
coefficients are central already,
it therefore lies in $Z_p(\SYmn)[[u^{-1}]]$,
which by Proposition \ref{Proposition ZpSYmm=ZYmm}.
However, the definition of $Z_p(\SYmn)$ immediately implies
$Z_p(\SYmn)[[u^{-1}]]\subset
Z_p(\Ymn)[[u^{-1}]]$,
as desired.
\end{proof}

\begin{Theorem}\label{final theorem}
The $p$-center $Z_p(\Ymn)$ is freely generated by
$\{s_{i,j}^{(rp)};~1 \leq i,j \leq m+n, r > 0,|i|+|j|=0\}$.
We have that $s_{i,j}^{(rp)} \in{\rm F}_{rp-p}\Ymn$ and
\begin{equation}\label{this}
\gr_{rp-p}s_{i,j}^{(rp)}=(e_{i,j} t^{r-1})^p
- \delta_{i,j} e_{i,j} t^{rp-p}.
\end{equation}
For $0 < r< p$, we have that $s_{i,j}^{(r)} = 0$.
For $r \geq p$ with $p \nmid r$, the central element $s_{i,j}^{(r)}$
belongs to ${\rm F}_{rp-p-1} Y_n$ and it may be expressed as a polynomial
in the elements $\{s_{i,j}^{(ps)};~0 < s \leq \lfloor r/p\rfloor\}$.
\end{Theorem}
\begin{proof}
Let $t'_{i,j}(u):=\sum_{r \geq 0}t_{i,j}'^{(r)}u^{-r}$.
By multiplying out the matrix products $T(u)=F(u)D(u)$ and $T(u)^{-1}=E(u)^{-1}D(u)^{-1}F(u)^{-1}$,
one obtain that $t_{i,j}'^{(r+1)}\in{\rm F}_r\Ymn$ and $\gr_r t_{i,j}'^{(r+1)}=-(-1)^{|i|}e_{i,j}x^r$.
By using Lemma \ref{Lemma: sij coeff in p-center} and passing to the associated graded algebra,
the rest of the proof is the same as in the non-super case \cite[Theorem 6.9]{BT18}, and will
be skipped.
\end{proof}

\bigskip
\noindent
\textbf{Acknowledgment.} We would like to thank Lewis Topley for giving a nice talk about
the modular finite $W$-algebras and shifted Yangians.
We also thank Yung-Ning Peng for helpful discussions.
This work is supported by the National Natural Science Foundation of China (Grant Nos.11801204, 11801394).

\end{document}